\pdfoutput=1
\documentclass[
	english
]{scrartcl}

\usepackage[T1]{fontenc}
\usepackage[utf8]{inputenc}

\usepackage{amsmath,amsfonts,amsthm,amssymb}
\usepackage{xcolor}
\usepackage[colorlinks,citecolor=blue]{hyperref}

\usepackage{preamble}
\usepackage{marginnote}

\usepackage{enumitem}
\setlist{itemsep=0em} 
\setlist[enumerate]{label=(\roman*)}

\usepackage[capitalise,nameinlink]{cleveref}

\usepackage{dsfont}

\usepackage{soul,cancel}

\usepackage{graphicx}

\usepackage{mathtools} 

\newif\ifbiber
\bibertrue

\ifbiber
\usepackage[
	style=authoryear-comp,
	sorting=nyt,
	url=false, 
	doi=true,                
	eprint=true,
	uniquename=allinit,        
	maxbibnames=99,          
	maxcitenames=2,
	uniquelist=minyear,
	dashed=false,            
	sortcites=false,          
	backend=biber,           
	safeinputenc,            
]{biblatex}

\DeclareCiteCommand{\cite}{%
	\ifbibmacroundef{cite:init}{}{\usebibmacro{cite:init}}\usebibmacro{prenote}%
}{%
	\usebibmacro{citeindex}%
	\printtext[bibhyperref]{\usebibmacro{cite}}%
}{%
	\ifbibmacroundef{cite:init}{\multicitedelim}{}%
}{%
	\usebibmacro{postnote}%
}%
\DeclareCiteCommand{\parencite}[\mkbibbrackets]{%
	\ifbibmacroundef{cite:init}{}{\usebibmacro{cite:init}}\usebibmacro{prenote}%
}{%
	\usebibmacro{citeindex}%
	\printtext[bibhyperref]{\usebibmacro{cite}}%
}{%
	\ifbibmacroundef{cite:init}{\multicitedelim}{}%
}{%
	\usebibmacro{postnote}%
}%

\let\cite\parencite

\DeclareNameAlias{sortname}{last-first}
\DeclareNameAlias{default}{last-first}

\else

\usepackage{doi,natbib}

\fi

\newcommand\norm[1]{\lVert#1\rVert}
\newcommand\bignorm[1]{\bigl\lVert#1\bigr\rVert}

\newcommand\abs[1]{\lvert#1\rvert}
\newcommand\bigabs[1]{\bigl\lvert#1\bigr\rvert}

\newcommand\Bigabs[1]{\Bigl\lvert#1\Bigr\rvert}

\newcommand\dual[2]{\langle #1, #2\rangle}
\newcommand\bigdual[2]{\bigl\langle #1, #2\bigr\rangle}
\newcommand\Bigdual[2]{\Bigl\langle #1, #2\Bigr\rangle}
\newcommand\scalarprod[2]{( #1, #2)}
\newcommand\bigscalarprod[2]{\bigl( #1, #2\bigr)}

\newcommand\N{\mathbb{N}}
\newcommand\Z{\mathbb{Z}}
\newcommand\R{\mathbb{R}}

\renewcommand\d{\mathrm{d}}

\newcommand{\weakly}{\rightharpoonup}
\newcommand{\weaklystar}{\stackrel\star\rightharpoonup}

\newcommand{\dom}{\operatorname{dom}}

\newcommand{\jxa}{_j^{\bar x_0, a_0}}
\newcommand{\jxg}{\texorpdfstring{_j^{x,g}}{jxg}}
\newcommand{\jxp}{_j^{\bar u,-\bar\varphi}}

\newcommand{\anni}{^\perp}

\newcommand{\dualspace}{^\star}
\newcommand{\adjoint}{^\star}

\renewcommand\AA{\mathcal{A}}

\newcommand\DD{\mathcal{D}}
\newcommand\HH{\mathcal{H}}
\newcommand\LL{\mathcal{L}}
\newcommand\TT{\mathcal{T}}

\newcommand\KK{\mathcal{K}}
\newcommand\MM{\mathcal{M}}
\newcommand\NN{\mathcal{N}}
\newcommand\ZZ{\mathcal{Z}}

\newcommand\Uad{U_{\mathrm{ad}}}

\DeclareMathAlphabet{\mathpzc}{OT1}{pzc}{m}{it}
\newcommand\oo{\mathpzc{o}}

\newcommand\dist{\operatorname{dist}}
\newcommand\conv{\operatorname{conv}}

\newcommand\proj{\operatorname{proj}}

\let\subseteq\subset

\newtheorem{theorem}{Theorem}[section]
\newtheorem{lemma}[theorem]{Lemma}

\newtheorem{assumption}[theorem]{Assumption}

\newtheorem{corollary}[theorem]{Corollary}
\newtheorem{remark}[theorem]{Remark}
\newtheorem{definition}[theorem]{Definition}

\crefname{assumption}{Assumption}{Assumptions}

\definecolor{darkgreen}{rgb}{0,0.5,0}
\definecolor{darkred}{rgb}{0.8,0,0}

\begin{document}
\title{Differential Sensitivity Analysis of Variational Inequalities with Locally Lipschitz Continuous Solution Operators\footnote{
This research was supported by the German Research Foundation (DFG) under grant numbers ME 3281/7-1 and WA 3636/4-1
 within the priority program ``Non-smooth and Complementarity-based Distributed Parameter
Systems: Simulation and Hierarchical Optimization'' (SPP 1962).}
}

\author{%
 	Constantin Christof%
 	\footnote{%
 		Technische Universität Dortmund,
 		Faculty of Mathematics,
 		LS X, 
 		44227 Dortmund,
 		Germany
 		}%
 	\and
 	Gerd Wachsmuth%
 	\footnote{%
 		Technische Universität Chemnitz,
 		Faculty of Mathematics,
 		Professorship Numerical Mathematics (Partial Differential Equations),
 		09107 Chemnitz,
 		Germany,
 		\url{http://www.tu-chemnitz.de/mathematik/part_dgl/people/wachsmuth/},
 		\email{gerd.wachsmuth@mathematik.tu-chemnitz.de}%
 	}%
 }
 \publishers{}
 \maketitle

 \begin{abstract}
This paper is concerned with the differential sensitivity analysis of variational inequalities
in Banach spaces whose solution operators satisfy a generalized Lipschitz condition.
We prove a sufficient criterion for the directional differentiability of the solution map
that turns out to be also necessary for elliptic variational inequalities in Hilbert spaces
(even in the presence of asymmetric bilinear forms, nonlinear operators and nonconvex functionals). 
In contrast to classical results, our method of proof does not rely on Attouch's theorem on
the characterization of Mosco convergence but is fully elementary. 
Moreover, our technique allows us to also study those cases where the variational inequality at hand is not uniquely solvable 
and where directional differentiability can only be obtained w.r.t.\ the weak or the weak-$\star$  topology of the underlying space. 
As tangible examples, we consider a variational inequality arising in elastoplasticity, the projection onto prox-regular sets,
and a bang-bang optimal control problem.
 \end{abstract}
 
 
\begin{msc}
	\mscLink{90C31},
	\mscLink{49K40},
	\mscLink{47J20}
\end{msc}

\section{Introduction}

The aim of this paper is to study the differentiability properties of the solution operator to a parametrized variational inequality (VI) of the form 
\begin{equation}
	\label{eq:VI}
	\bar x \in X,\qquad 
	\dual{A(p, \bar x)}{x - \bar x} + j(x) - j(\bar x) \ge 0
	\qquad\forall x \in X.
\end{equation}
Here, $X$ denotes a Banach space, $A$ is an operator into the topological dual  $X\dualspace$ of $X$,  $j : X \to (-\infty, \infty]$ 
is a proper function (i.e., $j \not \equiv \infty$), and $p$ is an element of some parameter space $P$
(the argument of the solution map). 
For the precise assumptions on the quantities in \eqref{eq:VI}, 
we refer to \cref{sec:abstract_analysis}. 

Note that VIs of the type \eqref{eq:VI} occur naturally as optimality conditions
for minimization problems of the form
\begin{equation}
	\label{eq:minimization_problem}
	\text{Minimize} \quad J(p,x) + j(x)
	\qquad\text{w.r.t.\ } x \in X.
\end{equation}
Indeed, if $j$ is convex and $J(p, \cdot) : X \to \mathbb{R}$ is convex and Gâteaux differentiable,
then \eqref{eq:VI} with $A(p, \cdot ) = \partial_x J(p, \cdot)$  is a necessary and sufficient optimality condition for \eqref{eq:minimization_problem}.

Differentiability results for special instances of the VI \eqref{eq:VI} 
or the minimization problem \eqref{eq:minimization_problem}
can be found frequently in the literature. 
Especially  the case where $X$ is a Hilbert space and where \eqref{eq:minimization_problem}
describes the metric projection onto a closed convex nonempty set $K$ 
(i.e., where $P=X$, $J(p, x) = \frac12\norm{x - p}_X^2$, and $j = \delta_K : X \to \{0,\infty\}$ is the indicator function of $K$) 
has been studied extensively throughout the years in a wide variety of different settings.
Exemplarily, we mention 
\cite{Zarantonello1971,Mignot1976,Haraux1977,FitzpatrickPhelps1982,Rockafellar1990,Shapiro1994,Noll1995,RockafellarWets1998,Levy1999,Shapiro2016}.
Results that cover cases where $j$ is not the indicator function of some set $K$ may be found, 
e.g., in \cite{Sokolowski1988,SokolowskiZolesio1992,Do1992,BorweinNoll1994} 
and the more recent \cite{DelosReyes2016,ChristofMeyer2016,Adly2017,Hintermueller2017,ChristofWachsmuth2017:2}.

The contributions that shed the most light on the mechanisms that underlie the sensitivity
analysis of VIs of the form \eqref{eq:VI} are probably \cite{Do1992,BorweinNoll1994,RockafellarWets1998} 
and \cite{Adly2017}. In these works, 
it is shown that the differentiability properties of the solution operator to \eqref{eq:VI} are directly related 
to the so-called second-order epi-differentiability of the  functional $j : X \to (-\infty, \infty]$, cf.\ \cref{def:second_order_derivative}. 
More precisely, in \cite[Chapter 13G]{RockafellarWets1998}, \cite[Theorems 3.9 and 4.3]{Do1992} 
and \cite[Proposition 6.3]{BorweinNoll1994} it is established that the directional differentiability 
of the solution operator to \eqref{eq:minimization_problem} in a point $p$ is equivalent to the (strong) second-order 
epi-differentiability of the functional $j$ in $\bar x$ and the proto\-differentiability of the subdifferential $\partial j$ in $\bar x$, respectively,  
provided $X$ is a Hilbert space, $J(p, x) = \frac12\norm{x - p}_X^2$, and $j$ is a convex and lower semicontinuous function. 
In \cite{Adly2017}, a similar (but only sufficient) criterion for the directional differentiability of 
the solution map is obtained for problems that are not only perturbed in the operator $A$ 
but also in the functional $j$, see \cite[Theorem~41]{Adly2017}. 

What the approaches in \cite{Do1992,BorweinNoll1994,RockafellarWets1998,Adly2017} 
have in common is that they rely heavily on rather involved 
concepts and theorems from set-valued and convex analysis and 
the theory of monotone operators, cf., e.g., the proof of \cite[Theorem 4.3]{Do1992}. 
In this paper, we will demonstrate that the majority of the results in  \cite{Do1992,BorweinNoll1994,RockafellarWets1998,Adly2017} 
can be reproduced and even extended using only elementary 
tools from functional analysis 
(the most complicated are the theorem of Banach-Alaoglu and Banach's fixed-point theorem). 
The main advantages and novel features of our approach are the following:

\begin{enumerate}
\item 
We can establish that the second-order epi-differentiability 
of the functional $j$ in $\bar x$ is sufficient for the directional differentiability
of the solution operator to \eqref{eq:VI} without making 
use of involved instruments from convex and set-valued analysis, 
see \cref{thm:sufficient_abstract} and the more tangible \cref{corollary:tangible}. 
We do not have to invoke, e.g., Attouch's theorem which is at the heart of the proofs in \cite{Do1992}.
We further emphasize that the proof of \cref{thm:sufficient_abstract} is shorter than one page
and its most complicated argument is the selection of a weak-$\star$ convergent subsequence.

\item 
Because of its simplicity, our analysis allows for various generalizations and extensions. 
In particular, it is also applicable when  \eqref{eq:VI} is not uniquely solvable, 
when $X$ is not a Hilbert space, 
when $j$ is not convex, 
and when the directional differentiability is only obtainable in 
the weak or the weak-$\star$ topology of the underlying space,
cf.\ the analysis in \cref{sec:abstract_analysis} and the examples in \cref{sec:examples}.

\item 
In the case of an elliptic variational inequality in a Hilbert space,
our approach yields the equivalence of the (strong) second-order epi-differentiability 
of $j$ in $\bar x$ and the directional differentiability of the solution operator to \eqref{eq:VI}  
even in the presence of nonlinear operators, asymmetric bilinear forms and nonconvex functionals (see \cref{th:elliptic}). 
We are thus able to extend \cite[Theorem 4.3]{Do1992} and \cite[Proposition 6.3]{BorweinNoll1994} 
(which require $A$ to be given by $A(p,x) := x-p$ and $j$ to be convex and lower semicontinuous, 
and which rely heavily on results for the classical Moreau-Yosida regularization) 
to cases where the VI at hand cannot be identified with a minimization problem of the form \eqref{eq:minimization_problem}
and where the method of proof in \cite{Do1992,BorweinNoll1994} cannot be employed.
\end{enumerate}

We hope that the self-containedness and conciseness 
of our approach make this paper in particular 
helpful for those readers who are interested in the sensitivity analysis of VIs 
of the form \eqref{eq:VI} but who are not familiar with, e.g., 
the concepts of graphical convergence and protodifferentiability.

We would like to point out that the ideas that our analysis is based on 
can also be used to obtain differentiability results for VIs that involve 
not only a parameter-dependent operator $A$ but also a parameter-dependent functional $j$.
In \cite{ChristofMeyer2016}, for example, 
our approach was used to study the directional differentiability of the solution map 
$L^\infty_+(\Omega) \times H^{-1}(\Omega) \to H_0^1(\Omega)$, $(c, p) \mapsto \bar x$, 
to a $H_0^1(\Omega)$-elliptic VI of the form 
\begin{equation*} 
\bar x \in H_0^1(\Omega), \qquad \dual{A(\bar x) - p}{x - \bar x}  +  \int_\Omega c\, k(x) \mathrm{d}\lambda   -  \int_\Omega c \, k(\bar x) \mathrm{d}\lambda \geq 0 \qquad \forall x \in H_0^1(\Omega),
\end{equation*}
where $k$ is the Nemytskii operator of a piecewise-smooth convex real-valued function. 
We restrict our analysis to perturbations in the operator $A$ since a unified description of 
the sensitivity analysis becomes rather involved when perturbations in the functional $j$ are considered, 
and since somewhat peculiar effects occur when the functional $j$ is manipulated. 
See, e.g., the results in \cite[Section 5]{ChristofMeyer2016} for some examples and \cite{Adly2017} 
where the approach of \cite{Do1992} is generalized to parameter-dependent functionals $j$. 

Before we begin with our analysis, we give a short overview of the contents and the structure of this paper.

In \cref{sec:abstract_analysis}, 
we study the differentiability properties of the solution operator to the VI \eqref{eq:VI} in an abstract setting. 
Here, we also motivate and introduce the notions of ``weak-$\star$ second subderivative'' (\cref{def:weak_star_subderivative}) and
``second-order epi-differentiability'' (\cref{def:second_order_derivative}) that are needed for our approach. 
The main results of \cref{sec:abstract_analysis}, \cref{thm:necessary_condition_1} and \cref{thm:sufficient_abstract},
yield that directional derivatives of the solution map to \eqref{eq:VI} are
themselves solutions to suitably defined variational inequalities 
and that the second-order epi-differentiability of $j$ is sufficient for the directional differentiability of the solution operator to \eqref{eq:VI}. 

In \cref{sec:tangible_corollary},
we state a self-contained corollary of \cref{thm:sufficient_abstract} that is more tangible than the results of \cref{sec:abstract_analysis}.
\cref{sec:tangible_corollary} further contains a criterion for second-order epi-differentiability that is of major importance not only for practical applications 
but also for the development of the theory. 

In \cref{sec:elliptic}, 
we consider the special case that $X$ is a Hilbert space and that $A$ is strongly monotone. 
In this situation, the sufficient differentiability criterion proved in \cref{sec:abstract_analysis} 
is also necessary and, as a consequence, sharp. 

In \cref{sec:examples}, we apply our results to three model problems. 
These examples are not covered by the classical theory
and, thus, highlight the broad applicability of our results.
The first problem is a variational inequality of the first kind with saddle-point structure that arises in elastoplasticity and has been studied, e.g., in \cite{HerzogMeyerWachsmuth2010:2}.
Second, we study the projection onto prox-regular sets.
Finally, we apply our theorems to bang-bang optimal control problems in the measure space $\MM(\Omega)$. 
The results that we obtain here underline that it makes sense to study the variational inequality \eqref{eq:VI} in a Banach space setting 
and that the generality of our approach is not only of theoretical interest but also of relevance in practice. 

Lastly, in \cref{sec:conclusion}, we summarize our findings and make some concluding remarks. 
 
\section{Sensitivity Analysis in an Abstract Setting}
\label{sec:abstract_analysis}

As already mentioned in the introduction, the aim of this paper is to study variational inequalities of the form \eqref{eq:VI},
i.e., problems of the type
\begin{equation*}
	\bar x \in X,\qquad 
	\dual{A(p, \bar x)}{x - \bar x} + j(x) - j(\bar x) \ge 0
	\qquad\forall x \in X.
\end{equation*}
Our standing assumptions on the quantities in \eqref{eq:VI} are as follows:

\pagebreak[2]

\begin{assumption}[Functional Analytic Setting]\hfill
	\label{asm:data}
	\begin{itemize}
		\item
			$X$ is the (topological) dual of a reflexive or separable Banach space $Y$.
		\item
			$p$ is an element of a normed vector space $P$ (the space of parameters).
		\item
			$j : X \to (-\infty, \infty]$ is a proper function (not necessarily convex).
		\item
			$A : P \times X \to Y$ is a mapping into the predual $Y$ of $X$. 
	\end{itemize}
\end{assumption}

Our main interest is in the differentiability properties of the (potentially set-valued) solution operator
\begin{equation*}
S : P \rightrightarrows X,\qquad p \mapsto \{\bar x \in X \mid \bar x \text{ solves \eqref{eq:VI} with parameter } p\}.
\end{equation*}
To study the latter in the greatest possible generality, 
we avoid discussing the solvability of the problem \eqref{eq:VI} and simply state the minimal assumptions
that the solutions to  \eqref{eq:VI}  have to satisfy for our sensitivity analysis to hold. 
Tangible examples (e.g., applications with elliptic variational inequalities in Hilbert spaces) 
will be addressed later on, cf.\ \cref{sec:tangible_corollary,sec:elliptic,sec:examples}. 			

\begin{assumption}[Standing Assumptions for the Sensitivity Analysis] 
\label{assumption:sensitivity}
We are given two families $\{q_t\}_{0 < t < t_0} \subset P$ and $\{\bar x_t\}_{0 \leq t < t_0} \subset X$, $t_0 > 0$, 
such that the following is satisfied:
	\begin{enumerate}
		\item
			\label{assumption:sensitivity:i}
			It holds $q_t \to q$ in $P$ for $t \searrow 0$ with some $q \in P$.
		\item
			\label{assumption:sensitivity:ii}
			It holds $\bar x_t \in S(t q_t)$ for all $0 < t < t_0$,
			$\bar x_0 \in S(0)$, and there exists a constant $L>0$ with
			\begin{equation}
				\label{eq:LipschitzEstimate}
				\norm{\bar x_t - \bar x_0}_X \leq L t \quad \forall t \in [0, t_0).
			\end{equation}		
		\item
			\label{assumption:sensitivity:iii}
			There exist bounded linear operators $A_p \in \LL(P, Y)$ and $A_x \in \LL(X, Y)$ 
			   such that the difference quotients $y_t := (\bar x_t - \bar x_0)/t$, $0 < t < t_0$, satisfy
			\begin{equation}
				\label{eq:taylor_A}
					A(t \, q_t , \bar x_0 + t \, y_t)
					=
					A(0,\bar x_0)  + t \, A_p q_t + t \, A_x y_t + r(t)
			\end{equation}
			with a remainder $r : (0,t_0) \to Y$ such that $\|r(t)\|_Y/t \to 0$ for $t \searrow 0$.
	\end{enumerate}
\end{assumption}

\begin{remark}\hfill
\begin{enumerate}
\item Instead of $\bar x_t \in S(t q_t)$ for all $0 < t < t_0$ and $\bar x_0 \in S(0)$ we could also assume  
$\bar x_t \in S(p + t q_t)$ for all $0 < t < t_0$ and $\bar x_0 \in S(p)$ with some fixed $p \in P$. 
Since such a $p$ can always be ``hidden'' by redefining $A$, we consider w.l.o.g.\ the case $p=0$.

\item The Lipschitz condition in \cref{assumption:sensitivity} \ref{assumption:sensitivity:ii} is, e.g., 
satisfied in case of an elliptic variational inequality in a Hilbert space, cf.\ \cref{sec:elliptic}.  
If \eqref{eq:VI} can be identified with a minimization problem of the form \eqref{eq:minimization_problem}, \eqref{eq:LipschitzEstimate} 
can further be recovered from a quadratic growth condition for the solution $\bar x_0$ of the unperturbed problem, cf.\ \cref{subsec:bang_bang}.

\item \cref{assumption:sensitivity} \ref{assumption:sensitivity:iii}
is, e.g., satisfied if  $A$ is  Fréchet differentiable in $(0, \bar x_0)$.
\end{enumerate}
\end{remark}

We emphasize that we do not say anything about the uniqueness of solutions in the above. 
We just assume that a family $\{\bar x_t\}_{0 \leq t < t_0}$ with the properties in \cref{assumption:sensitivity} exists. 
In what follows, our aim will be to prove necessary and sufficient conditions
for the weak-$\star$ and the strong convergence of the difference 
quotients $y_t$. Note that, if $q_t = q$ and if $S(tq)$ is a singleton for all $0 \leq t < t_0$, 
then the weak-$\star$ (respectively, strong) convergence of $y_t$ to some $y$  for $t \searrow 0$ 
is equivalent to the weak-$\star$ (respectively, strong) directional differentiability 
of the solution operator $S$ in the point $p=0$ in the direction $q$ with directional derivative $y$. 
To study the behavior of the difference quotients $\{  y_t\}_{0 < t < t_0} $, 
we make the following observation:

\begin{lemma}
	\label{lem:test_VI}
	The difference quotients $y_t$, $0 < t < t_0$, satisfy
	\begin{equation}
		\label{eq:VI_diffquot}
		\begin{aligned}
			\bigdual{A_p q_t + A_x y_t}{z - y_t}
			&+
			\frac12 \, \left ( \frac{j(\bar x_0 + t \, z) - j(\bar x_0) - t \, \dual{a_0}{z}}{t^2 / 2} \right )
			\\
			&-
			\frac12 \, \left ( \frac{j(\bar x_0 + t \, y_t) - j(\bar x_0) - t \, \dual{a_0}{y_t}}{t^2 / 2} \right )
			+
			\hat r(t) \, \norm{z - y_t}_X
			\ge0
		\end{aligned}
	\end{equation}
	for all $z \in X$.
	Here, $a_0 := -A(0, \bar x_0)$ and  $\hat r(t) :=  \norm{r(t)}_Y / t$, so that $\hat r(t)  = \oo(1)$ as $t \searrow 0$.
\end{lemma}
\begin{proof}
	Since $\bar x_t = \bar x_0 + t y_t$ solves \eqref{eq:VI} with $p = t \, q_t$ and because of \eqref{eq:taylor_A}, it holds
	\begin{align*}
		0
		&\le
		\bigdual{A(t \, q_t, \bar x_t)}{\bar x_0 + t \, z - \bar x_t} + j(\bar x_0 + t \, z) - j(\bar x_t) \\
		&=
		\bigdual{-a_0 + t \, A_p q_t + t \, A_x y_t + r(t)}{t \, (z - y_t)} + j(\bar x_0 + t \, z) - j(\bar x_0 + t \, y_t)
	\end{align*}
	for all $z \in X$, 
	where $\norm{r(t)}_Y / t \to 0$ as $t \searrow 0$.
	Dividing by $t^2$, rearranging terms and using that $j(\bar x_0) \in \R$, \eqref{eq:VI_diffquot} follows immediately 
	with $\hat r(t) = \norm{r(t)}_Y / t$.
\end{proof}

Note that from $\bar x_0 \in S(0)$ and \eqref{eq:VI}, we obtain that 
$\bar x_0$ and $a_0 := -A(0, \bar x_0)$ satisfy $\bar x_0 \in \dom(j) := \{x \in X \mid j(x) \in \mathbb{R}\}$ and
$a_0 \in \partial j(\bar x_0)$, where
\begin{equation*}
	\partial j(x)
	:=
	\{
		g \in Y
		\mid
		j(z) \ge j(x) + \dual{g}{z- x} \; \forall z \in X
	\}\quad \forall x \in X.
\end{equation*}
This shows that the bracketed expressions in \eqref{eq:VI_diffquot} may be interpreted as second-order difference quotients
in which the (possibly nonexistent) derivative of $j$ at $\bar x_0$ is replaced with an element of the subdifferential $\partial j(\bar x_0)$ 
(where we use the term subdifferential  somewhat loosely here since $j$ is not assumed to be convex). The structure of \eqref{eq:VI_diffquot} 
motivates the following definition.

\begin{definition}[Weak-\texorpdfstring{$\star$}{*} Second Subderivative]
	\label{def:weak_star_subderivative}
	Let $x \in  \dom(j)$ and $g \in Y$ be given.
	Then the (weak-$\star$) second subderivative  $Q\jxg : X \to [-\infty,\infty]$ of $j$ in $x$ for $g$ is defined by
	\begin{equation*}
		Q\jxg(z)
		:=
		\inf
		\biggh\{\}{
			\liminf_{n \to \infty} \frac{j(x + t_n \, z_n) - j(x) - t_n \dual{g}{z_n}}{t_n^2/2}
			\mid
			t_n \searrow 0,
			z_n \weaklystar z
		}
		.
	\end{equation*}
\end{definition}
The notion of second subderivatives goes (at least to the authors' best knowledge) 
back to Rockafellar who introduced the concept in finite dimensions in 1985, see \cite{ROCKAFELLAR1985167}. 
Since then second subderivatives have appeared frequently in the literature, 
although under different names.  
\cite{Do1992} and \cite{Noll1995}, for example, 
use a construction analogous  to that in \cref{def:weak_star_subderivative} in the Hilbert space setting 
and call the resulting functional second-order epi-derivative and second-order Mosco derivative, respectively.
We remark that the epigraph of $Q\jxg$ can be identified with an appropriately defined Kuratowski limit 
of the epigraphs of the difference quotient functions appearing in \eqref{eq:VI_diffquot}, cf.\ \cite[Section 1]{Do1992}. 

The next lemma collects some basic properties of the functional $Q\jxg$.
\begin{lemma}
	\label{lem:basic_properties}
	Let $x \in  \dom(j)$ and $g \in Y$ be arbitrary but fixed.
	Then, it holds
	$Q\jxg(\alpha \, z) = \alpha^2 \, Q\jxg(z)$ for all $\alpha > 0$ and all $z \in X$
	and $Q\jxg(0) \le 0$.
	Moreover,
	if $g \in \partial j(x)$, then
	$Q\jxg(z) \ge 0$ for all $z \in X$
	and
	$Q\jxg(0) = 0$.
\end{lemma}
\begin{proof}
	The first formula follows from a simple scaling argument. 
	In the case
	$g \in \partial j(x)$,
	the nonnegativity of $Q\jxg$ follows from the definition of $\partial j$.
	The formulas for
	$Q\jxg(0)$ follow from the choice $z_n = 0$.
\end{proof}

In the case that $g \in \partial j(x)$,
\cref{lem:basic_properties} implies in particular that
$Q\jxg$ is proper
and that
the domain of the second subderivative $Q\jxg$ 
is a pointed cone (where ``pointed'' means that the cone contains the origin). 
In what follows, we will call this cone the \emph{reduced critical cone} $\KK\jxg$, i.e.,
\begin{equation*}
	\KK\jxg
	:=
	\dom \bigh(){ Q\jxg }
	=
	\bigh\{\}{z \in X \mid Q\jxg(z) < +\infty }.
\end{equation*}
The motivation behind this naming convention will become clear in \cref{lem:more_properties_of_Q} \ref{item:more_properties_of_Q_3} 
and the examples in \cref{sec:examples}. 

We are now in the position to prove that a limit point $y$ of the difference quotients $y_t$ has to be the solution 
of  a 
certain VI that involves the second subderivative.

\begin{theorem}[Necessary Condition for Limit Points of the Difference Quotients \texorpdfstring{$y_t$}{}]
\label{thm:necessary_condition_1}
	Suppose that there exists a $y \in X$ such that the difference quotients $y_t$ satisfy
	$y_t \weaklystar y$ in $X$
	and
	$A_x y_t \to A_x y$ in $Y$
	for $t \searrow 0$.
	Then,
	$y$ satisfies	
	\begin{equation}
		\label{eq:linearized-VI}
		\dual{A_p q + A_x y}{z - y}
		+ \frac12 Q_j^{\bar x_0, a_0}(z)
		- \frac12 Q_j^{\bar x_0, a_0}(y)
		\ge
		0
		\qquad\forall z \in X
	\end{equation}
	and it holds
	\begin{equation}
		\label{eq:identity_Q_y}
		Q_j^{\bar x_0, a_0}(y)
		=
		-\dual{A_p q + A_x y}{y}
		<
		+\infty
	\end{equation}
	as well as
	\begin{equation}
		\label{eq:recovery_sequence}
		Q_j^{\bar x_0, a_0}(y)
		=
		\lim_{t \searrow 0} \frac{j(\bar x_0 + t \, y_{t}) - j(\bar x_0) - t \dual{a_0}{y_{t}}}{t^2/2}
		.
	\end{equation}
\end{theorem}

\begin{proof}
	We first prove \eqref{eq:linearized-VI}.
	Let $z \in \KK\jxa$ and $\varepsilon > 0$ be given. Then, the definitions of  $\smash{Q\jxa(z)}$ and $\smash{\KK\jxa}$ yield that there exist
	sequences $z_n \weaklystar z$ and $t_n \searrow 0$
	with
	\begin{equation*}
		Q\jxa(z) 
		\le
		\lim_{n \to \infty} \frac{j(\bar x_0 + t_n \, z_n) - j(\bar x_0) - t_n \dual{g}{z_n}}{t_n^2/2}
		\le Q\jxa(z) + \varepsilon < \infty
		.
	\end{equation*}
	From \eqref{eq:VI_diffquot} with $t_n$ and $z_n$, we infer
	\begin{align*}
		\bigdual{A_p q_{t_n} + A_x y_{t_n}}{z_n - y_{t_n}}
		&+
		\frac12 \, \frac{j(\bar x_0 + t_n \, z_n) - j(\bar x_0) - t_n \, \dual{a_0}{z_n}}{t_n^2 / 2}
		\\
		&-
		\frac12 \, \frac{j(\bar x_0 + t_n \, y_{t_n}) - j(\bar x_0) - t_n \, \dual{a_0}{y_{t_n}}}{t_n^2 / 2}
		+
		\hat r_n
		\ge0
	\end{align*}
	with $\hat r_n \to 0$ as $n \to \infty$.
	Passing to the limit $n \to \infty$ in the above, we find
	\begin{align*}
		0
		&\le
		\bigdual{A_p q + A_x y}{z - y}
		+
		\frac12 \, Q\jxa(z) + \varepsilon
		-
		\liminf_{n \to \infty} \frac12 \, \frac{j(\bar x_0 + t_n \, y_{t_n}) - j(\bar x_0) - t_n \, \dual{a_0}{y_{t_n}}}{t_n^2 / 2}
		\\
		&\le
		\bigdual{A_p q + A_x y}{z - y}
		+
		\frac12 \, Q\jxa(z) + \varepsilon
		-
		\frac12 \, Q\jxa(y).
	\end{align*}
	Letting $\varepsilon \searrow 0$, \eqref{eq:linearized-VI} now follows immediately. 

	In order to prove \eqref{eq:identity_Q_y},
	we first note that \eqref{eq:linearized-VI} with $z = 0 \in \KK\jxa$ yields $Q\jxa(y) < +\infty$.
	Choosing $z = s \, y$ with arbitrary $s \ge 0$ in \eqref{eq:linearized-VI} and using the positive homogeneity
	of $Q\jxa$, we further find that
	\begin{equation*}
		(s - 1) \, 
		\bigdual{A_p q + A_x y}{y}
		+
		\frac{s^2 - 1}{2} \, Q\jxa(y) \ge 0
		\qquad\forall s \ge 0.
	\end{equation*}
	Dividing this inequality by $s - 1$ and passing to the limits $s \searrow 1$ and $s \nearrow 1$, we obtain \eqref{eq:identity_Q_y} as claimed. 

	It remains to check \eqref{eq:recovery_sequence}.
	To this end, we fix a sequence $\{t_n\} \subset \R^+$ with $t_n \searrow 0$ and consider for $s>0$ the function
	\begin{equation*}
		\Theta(s)
		:=
		\limsup_{n \to \infty}
		\frac{j(\bar x_0 + s \, t_n \, y_{s \, t_n}) - j(\bar x_0) - s \, t_n \dual{a_0}{y_{s \, t_n}}}{(s \, t_n)^2/2}
		\ge
		Q\jxa(y)
		.
	\end{equation*}
	For arbitrary $s_1, s_2 > 0$,  \eqref{eq:VI_diffquot} with $t = s_1 \, t_n$, $z = \frac{s_2}{s_1} \, y_{s_2 \, t_n}$ yields
	\begin{align*}
		&
		\Bigdual{A_p q_{s_1\,t_n} + A_x y_{s_1\,t_n}}{\frac{s_2}{s_1} \, y_{s_2 \, t_n} - y_{s_1\,t_n}}
		\\&\qquad
		+
		\frac12 \, \frac{j(\bar x_0 + s_2 \, t_n \, y_{s_2 \, t_n}) - j(\bar x_0) - s_2 \, t_n \, \dual{a_0}{y_{s_2 \, t_n}}}{(s_1 \, t_n)^2 / 2}
		\\&\qquad
		-
		\frac12 \, \frac{j(\bar x_0 + s_1 \, t_n \, y_{s_1\,t_n}) - j(\bar x_0) - s_1 \, t_n \, \dual{a_0}{y_{s_1\,t_n}}}{(s_1 \, t_n)^2 / 2}
		+
		\hat r_n
		\ge0,
	\end{align*}
	where $\hat r_n \to 0$ as $n \to \infty$.
	Passing to the limit in the above (with a suitable subsequence), using \eqref{eq:identity_Q_y} and multiplying by $2$, we find
	\begin{equation}
		\label{eq:relation_s1_s2}
		2 \, \left ( 1 - \frac{s_2}{s_1} \right ) \, Q\jxa(y)
		+ \frac{s_2^2}{s_1^2} \, \Theta(s_2) - \Theta(s_1)
		\ge
		0.
	\end{equation}
	The same arguments with $s_2 = 0$ yield
	\begin{equation}
	\label{eq:inductionbasis}
		2 \, Q\jxa(y)
		- \Theta(s_1)
		\ge
		0\quad \forall s_1 >0.
	\end{equation}
	We will now prove that
	\begin{equation}
		\label{eq:induction_theta}
		\Theta(s) \le \frac{m+1}{m} \, Q\jxa(y)
		\qquad\forall s > 0 \quad \forall m \in \N.
	\end{equation}
	To obtain \eqref{eq:induction_theta}, we use induction over $m$.
	For $m = 1$,  \eqref{eq:induction_theta} is equivalent to \eqref{eq:inductionbasis} so there is nothing to prove. 
	For the induction step $m \mapsto m+1$, we choose $s_1 = s$ and $s_2 = \frac{m}{m+1} \, s$ in \eqref{eq:relation_s1_s2}. 
	This yields 
	\begin{align*}
		\Theta(s)
		&\le
		2 \, \Bigh(){ 1 - \frac{m}{m+1}} \, Q\jxa(y) + \frac{m^2}{(m+1)^2} \, \Theta(s_2)
		\\
		&\le
		2 \, \Bigh(){ 1 - \frac{m}{m+1}} \, Q\jxa(y) + \frac{m^2}{(m+1)^2} \, \frac{m+1}{m} \, Q\jxa(y)
		\le \frac{m+2}{m+1} \, Q\jxa(y),
	\end{align*}
	where in the second estimate we have used the induction hypothesis.
	Hence, \eqref{eq:induction_theta} is valid and the induction is complete. Letting $m \to \infty$ in \eqref{eq:induction_theta}, we   arrive at
	$\Theta(s) \le Q\jxa(y)$ for all $s>0$.
	With $s = 1$, we obtain in particular
	\begin{equation*}
		Q\jxa(y)
		\le
		\liminf_{n \to \infty} \frac{j(\bar x_0 + t_n \, y_{t_n}) - j(\bar x_0) - t_n \dual{a_0}{y_{t_n}}}{t_n^2/2}
		\le
		\Theta(1)
		\le
		Q\jxa(y).
	\end{equation*}
	Since $\{t_n\}$ was arbitrary, \eqref{eq:recovery_sequence} now follows immediately and the proof is complete. 
\end{proof}

\begin{remark}~
\begin{enumerate}
\item 
The assumptions $y_t \weaklystar y$ in $X$ and $A_x y_t \to A_x y$ in $Y$ in \cref{thm:necessary_condition_1} are satisfied in two interesting situations:
Firstly, if $y_t$ converges even strongly to $y$ and secondly if $\smash{y_t \weaklystar y}$ and if $A_x$ is weakly-$\star$ completely continuous. 
We will see in \cref{sec:elliptic,sec:examples} 
that both these cases appear in practice (the first one in the Hilbert space setting, the second one in case of our bang-bang example).

\item Assume for the moment that $q_t = q$, that $S(tq)$ is a singleton for all $0 \leq t < t_0$, and
that $S$ is (strongly) directionally differentiable in $p=0$ in the direction $q$ with derivative $y$.
Then, \cref{thm:necessary_condition_1} implies that $y$ has to be a solution to \eqref{eq:linearized-VI} and
that the difference quotients $y_t$ have to be a recovery sequence
for the weak-$\star$ second subderivative $Q_j^{\bar x_0, a_0}(y)$, see \eqref{eq:recovery_sequence}.
These necessary conditions for directional derivatives (that also apply when $S$ is only directionally differentiable in some directions)
have, at least to the authors' best knowledge, not been known before. 
\end{enumerate}
\end{remark}

The above observation that
the difference quotients $y_t$ provide a recovery sequence
motivates the following definition.

\begin{definition}[Second-Order Epi-Differentiability]
	\label{def:second_order_derivative}
	Let $x \in  \dom(j)$ and $g \in Y$ be given.
	The functional $j$ is said to be weakly-$\star$ twice epi-differentiable (respectively, strictly twice epi-differentiable, respectively, strongly twice epi-differentiable) 
	in $x$ for $g$ in a direction $z \in X$,
	if for all $\{t_n\}\subset \R^+$ with $t_n \searrow 0$ there exists a sequence $z_n$ satisfying $z_n \weaklystar z$ 
	(respectively, $z_n \weaklystar z$ and $\|z_n\|_X \to \|z\|_X$, respectively, $z_n \to z$) and 
	\begin{equation}
	\label{eq:recovery_sequence_def}
		Q\jxg(z)
		=
		\lim_{n \to \infty} \frac{j(x + t_n \, z_n) - j(x) - t_n \dual{g}{z_n}}{t_n^2/2}
		.
	\end{equation}
	The functional $j$ is called weakly-$\star$/strictly/strongly twice epi-differentiable in $x$ for $g$ if it is  
	weakly-$\star$/strictly/strongly twice epi-differentiable in $x$ for $g$ in all directions $z \in X$.
\end{definition}

\begin{remark}~
\label{remark:secondorderepidifferentiability}
\begin{enumerate}
\item We emphasize that the prefixes ``weakly-$\star$'', ``strictly'' and ``strongly'' 
in \cref{def:second_order_derivative} refer to the mode of 
convergence of the recovery sequence. In all cases, the considered second subderivative is that in  \cref{def:weak_star_subderivative}.

\item If $X$ is reflexive, then strong second-order epi-differentiability is equivalent to the Mosco epi-convergence 
of the sequence of second-order difference quotient functions appearing in \eqref{eq:VI_diffquot}, see, e.g., \cite[Section 2]{Do1992}.

\item Note that $j$ is weakly-$\star$/strictly/strongly twice epi-differentiable in 
an $x \in \dom(j)$ for a $g \in \partial j(x)$ if and only if $j$ is
 weakly-$\star$/strictly/strongly twice epi-differentiable in $x$ for $g$ in all directions $z \in \KK\jxg$. 
This follows from the fact that for all $z \in X \setminus \KK\jxg$, recovery sequences can trivially be found (just choose, e.g., $z_n := z$). 
\label{remark:secondorderepidifferentiability:iii}
\end{enumerate}
\end{remark}

If $j$ is twice epi-differentiable in a point $x \in  \dom(j)$  for some $g \in \partial j(x)$, then  
$Q\jxg$ enjoys additional properties as the following lemma shows.
\begin{lemma} 
	\label{lem:more_properties_of_Q}
	Let $x \in  \dom(j)$ and $g \in \partial j(x)$ be given.
	\begin{enumerate}
		\item
			\label{item:more_properties_of_Q_1}
			If $j$ is convex and weakly-$\star$ twice epi-differentiable in $x$ for $g$,
			then $Q\jxg$ is convex.
		\item
			\label{item:more_properties_of_Q_2}
			If $j$ is strictly twice epi-differentiable in $x$ for $g$,
			then $Q\jxg$ is weakly-$\star$ sequentially lower semicontinuous.
		\item
			\label{item:more_properties_of_Q_3}
			If $j$ is Hadamard directionally differentiable in $x$ and strongly twice epi-differentiable in $x$ for $g$, then it holds
			$j'(x; z) = \dual{g}{z}$ for all $z \in \KK\jxg$.
	\end{enumerate}
\end{lemma}
\begin{proof}
	We first prove \ref{item:more_properties_of_Q_1}:
	Let $z, \hat z \in \KK\jxg$ and $\lambda \in [0,1]$ be given, 
	and let $\{t_n\} \subset \R^+$ be an arbitrary but fixed sequence with $t_n \searrow 0$.
	Then, the definition of weak-$\star$ second-order epi-differentiability implies that 
	there exist recovery sequences $z_n, \hat z_n$ with $z_n \weaklystar z$ and $\hat z_n \weaklystar \hat z$ 
	such that \eqref{eq:recovery_sequence_def} holds for $z$ and $\hat z$, respectively.
	Using these sequences, the convexity of $j$ and the definition of $\smash{Q\jxg\bigh(){ \lambda \, z + (1-\lambda) \, \hat z}}$, we may compute
	\begin{align*}
		&\lambda \, Q\jxg(z) + (1-\lambda) \, Q\jxg(\hat z)
		\\&\qquad
		=
		\lim_{n \to \infty}
		\frac{\lambda \, j(x + t_n \, z_n) + (1-\lambda)\,j(x + t_n \, \hat z_n) - j(x) - t_n \, \dual{g}{\lambda \, z_n + (1-\lambda) \, \hat z_n}}{t_n^2 / 2}
		\\&\qquad
		\ge
		\liminf_{n \to \infty}
		\frac{j\bigh(){x + t_n \, (\lambda \, z_n + (1-\lambda) \, \hat z_n)} - j(x) - t_n \, \dual{g}{\lambda \, z_n + (1-\lambda) \, \hat z_n}}{t_n^2 / 2}
		\\&\qquad
		\ge
		Q\jxg\bigh(){ \lambda \, z + (1-\lambda) \, \hat z}.
	\end{align*}
	This establishes \ref{item:more_properties_of_Q_1}.

	To obtain \ref{item:more_properties_of_Q_2}, we consider an arbitrary but fixed $z \in X$ and a sequence $z_k$ with $\smash{z_k \weaklystar z}$. 
	We  assume w.l.o.g.\ that $ \liminf_{k \to \infty} Q\jxg(z_k) = \lim_{k \to \infty} Q\jxg(z_k) \in \R $
	(if it holds $\liminf_{k \to \infty} Q\jxg(z_k) = \infty$, then the claim is vacuously true). 
	Suppose for the time being that $Y$ is separable with a countable dense subset $\{w_i\}_{i \in \N}$,
	let $\{t_n\} \subset \R^+$ be some sequence with $t_n \searrow 0$,
	and let 	$\{z_{k,n}\}_{n\in\N}$ be recovery sequences for the $z_k$ as in the definition of the strict second-order 
	epi-differentiability. Then, we may find a strictly increasing sequence $\{N_k\}$ such that
	\begin{equation}
		\label{eq:property_zkn_first}
		\sum_{i = 1}^k
		\abs{\dual{w_i}{z_{k,n} - z_k}}
		+
		\bigabs{\norm{z_{k,n}}_X - \norm{z_k}_X}
		+
		\Bigabs{
			Q\jxg(z_k)
			-
			\frac{j(x + t_n z_{k,n}) - j(x) - t_n \dual{g}{z_{k,n}}}{t_n^2/2}
		}
		\le
		\frac1k
	\end{equation}
	holds for all $n \ge N_k$ and all $k \in \N$. Redefine $N_1 := 1$ and set $k_n := \sup\{ k \in \N \mid  n \geq N_{k}\}$ for all $n \in \N$.
	Then, it holds $n \geq N_{k_n}$ for all $n$ by definition, $k_n \in \N$ for all $n$ by the strict monotonicity of $\{N_k\}$ and $k_n \to \infty$ monotonously for $n \to \infty$.
	The latter implies in tandem with \eqref{eq:property_zkn_first}  that $\hat z_n := z_{k_n,n}$
	satisfies $\dual{w_i}{\hat z_n - z_{k_n}}  \to 0$ for all $i \in \N$, $\abs{\norm{\hat z_n}_X - \norm{z_{k_n}}_X}  \to 0$ and 
	\begin{equation*}
		\Bigabs{
			\frac{j(x + t_n \, \hat z_n) - j(x) - t_n \dual{g}{\hat z_n}}{t_n^2/2}
			-
			Q\jxg(z_{k_n})
		}
		\to
		0
	\end{equation*}
	 as $n \to \infty$.
	From $z_{k_n} \weaklystar z$ and the boundedness of the norms $\norm{z_{k_n}}_X$, 
	we now obtain  
	$\hat z_n \weaklystar z$ 
	with
	\begin{equation*}
		\lim_{k \to \infty} Q\jxg(z_k)
		=
		\lim_{n \to \infty} Q\jxg(z_{k_n})
		=
		\liminf_{n \to \infty} \frac{j(x + t_n \, \hat z_n) - j(x) - t_n \dual{g}{\hat z_n}}{t_n^2/2}
		\ge
		Q\jxg(z)
		.
	\end{equation*}
	This establishes \ref{item:more_properties_of_Q_2} in the case that $Y$ is separable.
	If $Y$ is not separable but reflexive, we can use standard arguments as employed, e.g., in
	\cite[Proof of Theorem~6.24]{Kuttler1997} 
	to resort to the separable case.  
	
	It remains to prove \ref{item:more_properties_of_Q_3}. 
	To this end, suppose that a $z \in \KK\jxg$ is given and that $z_n$ is a recovery sequence for some $\{t_n\}$ with $t_n \searrow 0$
	as in the definition of the strong second-order epi-differentiability. Then, the Hadamard directional differentiability and the finiteness of $Q\jxg(z)$ yield
	\begin{equation*}
		0
		=
		\lim_{n \to \infty} \frac{j(x + t_n \, z_n) - j(x) - t_n \dual{g}{z_n}}{t_n}
		=
		j'(x; z) - \dual{g}{z}
		.
		\qedhere
	\end{equation*}
\end{proof}

\begin{remark}
If $j$ is Hadamard directionally differentiable in $\bar x_0$ and $a_0 := -A(0, \bar x_0)$, 
then it is easy to check that $\bar x_0 \in S(0)$ implies
\begin{equation}
\label{eq:nec_opt_condition}
j'(\bar x_0; z) - \left \langle a_0, z \right \rangle \geq 0\qquad \forall z \in X.
\end{equation}
In the case that the VI \eqref{eq:VI} arises from a minimization problem of the form \eqref{eq:minimization_problem},
\eqref{eq:nec_opt_condition} is precisely the necessary optimality condition of first order.
\cref{lem:more_properties_of_Q} \ref{item:more_properties_of_Q_3} shows that, under the assumptions of Hadamard directional differentiability
and strong second-order epi-differentiability, 
all elements of the set $\smash{\KK\jxa}$ satisfy the necessary condition \eqref{eq:nec_opt_condition} with equality. 
The set $\smash{\KK\jxa}$ is thus contained in what is typically referred to as the critical cone. 
We point out that the latter inclusion is in general strict, cf.\ the examples in \cref{sec:examples}. 
It therefore makes sense to call the set $\smash{\KK\jxa}$ the reduced critical cone. 
\end{remark}

We are now in the position to state the main theorem of this section.
It establishes that the second-order epi-differentiability of $j$ and the uniqueness of solutions to \eqref{eq:linearized-VI} 
are sufficient for the weak-$\star$ convergence of the difference quotients $y_t$. 

\begin{theorem}[Sufficient Condition for the Convergence of the Difference Quotients]
	\label{thm:sufficient_abstract}
	Suppose that one of the following conditions is satisfied.
	\begin{enumerate}
		\item
			\label{item:sufficient_abstract_1}
			$j$ is weakly-$\star$ twice epi-differentiable in $\bar x_0$ for $a_0$ and $A_x$ is weakly-$\star$ completely continuous in the sense that
			 $y_n \weaklystar y$ in $X$ implies $A_x y_n \to A_x y$ in $Y$.
		\item
			\label{item:sufficient_abstract_2}
			$j$ is strongly twice epi-differentiable in $\bar x_0$ for $a_0$ and $A_x$ is such that $y_n \weaklystar y$ in $X$ implies $A_x y_n \weakly A_x y$ in $Y$ 
			and $\liminf_{n \to \infty } \dual{A_x y_n}{y_n} \geq \dual{A_x y}{y}$. 
	\end{enumerate}
	Then, the sequence of difference quotients $y_t$ has at least one weak-$\star$ accumulation point for $t \searrow 0$, and if $y \in X$  is such an accumulation point, then it holds 
	\begin{equation}
		\label{eq:VI_derivative}
		\dual{A_p q + A_x y}{z - y}
		+ \frac12 Q_j^{\bar x_0, a_0}(z)
		- \frac12 Q_j^{\bar x_0, a_0}(y)
		\ge
		0
		\qquad\forall z \in X
	\end{equation}
	and $\dual{A_x y_{t_n}}{y_{t_n}} \to \dual{A_x y}{y}$ for every sequence $\{t_n\} \subset \R^+$ with $t_n \searrow 0$ and $y_{t_n} \weaklystar y$.
	If, moreover, \eqref{eq:VI_derivative} admits at most one solution, then there exists a unique $y \in X$ 
	with $y_t \weaklystar y$ and $\dual{A_x y_{t}}{y_{t}} \to \dual{A_x y}{y}$ for $t \searrow 0$, and this limit $y$ is a solution to \eqref{eq:VI_derivative}.
\end{theorem}

\begin{proof}
	Since the family of difference quotients $\{y_t\}$ is bounded by \eqref{eq:LipschitzEstimate}, 
	the existence of a weak-$\star$ accumulation point is a direct consequence of the theorem of
	Banach-Alaoglu.
	Consider now an arbitrary but fixed $y \in X$ that satisfies $y_n := y_{t_n} \weaklystar y$ for some 
	$\{t_n\} \subset \R^+$ with $t_n \searrow 0$ and let  $z \in \KK\jxa$ be given.
	Then, the definitions of weak-$\star$ and strong second-order epi-differentiability imply that
	in both cases 
	\ref{item:sufficient_abstract_1}
	and
	\ref{item:sufficient_abstract_2}
	we can find a recovery sequence $\{z_n\}$ with
	\begin{equation*}
		z_n \weaklystar z,
		\quad
		\dual{A_x y_n}{z_n} \to \dual{A_x y}{z},
		\quad
		Q\jxa(z)
		=
		\lim_{n \to \infty} \frac{j(\bar x_0 + t_n \, z_n) - j(\bar x_0) - t_n \dual{a_0}{z_n}}{t_n^2/2}.
	\end{equation*}
	Using the sequence $z_n$ in \eqref{eq:VI_diffquot}, we find that
	\begin{align*}
		\bigdual{A_p q_n + A_x y_n}{z_n - y_n}
		&+
		\frac12 \, \frac{j(\bar x_0 + t_n \, z_n) - j(\bar x_0) - t_n \, \dual{a_0}{z_n}}{t_n^2 / 2}
		\\
		&-
		\frac12 \, \frac{j(\bar x_0 + t_n \, y_n) - j(\bar x_0) - t_n \, \dual{a_0}{y_n}}{t_n^2 / 2}
		+
		\hat r_n \, \norm{z_n - y_n}_X
		\ge0,
	\end{align*}
	where $q_n := q_{t_n}$ and where $\hat r_n$ is a remainder with  $\hat r_n \searrow 0$ for $n \to \infty$.
	Letting $n \to \infty$ in the above, it follows
	\begin{align*}
		&
		\bigdual{A_p q + A_x y}{z}
		- 
		\bigdual{A_p q}{y}
		+
		\frac12\,Q\jxa(z)
		\\
		&\qquad
		\ge
		\limsup_{n \to \infty}
		\left (
		\bigdual{ A_x y_n}{ y_n}
		+
		\frac12 \, \frac{j(\bar x_0 + t_n \, y_n) - j(\bar x_0) - t_n \, \dual{a_0}{y_n}}{t_n^2 / 2}
		\right)
		\\
		&\qquad
		\ge
		\limsup_{n \to \infty}
		\bigdual{  A_x y_n}{ y_n}
		+
		\frac12\,Q\jxa(y)
		\\
		&\qquad
		\ge
		\liminf_{n \to \infty}
		\bigdual{ A_x y_n}{ y_n}
		+
		\frac12\,Q\jxa(y)
		\ge
		\bigdual{ A_x y}{y}
		+
		\frac12\,Q\jxa(y).
	\end{align*}
	Hence, $y \in \KK\jxa$ and $y$ solves \eqref{eq:VI_derivative}.
	Moreover, by using the test function $z = y$ in the above chain of inequalities, we obtain 
	$\dual{A_x y_n}{y_n} \to \dual{A_x y}{y}$.
	This proves the first claim.

	Suppose now that \eqref{eq:VI_derivative} admits at most one solution. 
	Then, the boundedness of the family $\{y_t\}$ implies that for every sequence $\{t_n\} \subset \R^+$ with $t_n \searrow 0$
	a subsequence of $\{y_{t_n}\}$ converges weakly-$\star$.
	From the first part of the theorem and the fact that \eqref{eq:VI_derivative} can have at most one solution, we obtain
	that the weak-$\star$ limit point is unique.
	A standard argument now shows that
	the entire sequence $y_t$ has to be weakly-$\star$ convergent to  $y$ 
	with $\dual{A_x y_{t}}{y_{t}} \to \dual{A_x y}{y}$ for $t \searrow 0$.
	This completes the proof. 
\end{proof}

Note that, as a byproduct of our sensitivity analysis, we obtain that  \eqref{eq:VI_derivative} always admits
a solution $y \in X$ in the situation of \cref{thm:sufficient_abstract}.

\section{A Tangible Corollary and Some Helpful Results}
\label{sec:tangible_corollary}

To make the results of \cref{sec:abstract_analysis} more accessible, we state the following self-contained corollary of \cref{thm:sufficient_abstract} that covers the case where
the solution operator $S : P \rightrightarrows X$ satisfies a generalized local Lipschitz condition. 

\begin{corollary}[Directional Differentiability in the Case of Local Lipschitz Continuity]
\label{corollary:tangible}
Let $S : P \rightrightarrows X$ denote the (potentially set-valued) solution operator of the VI
\begin{equation*}
	\bar x \in X,\qquad 
	\dual{A(p, \bar x)}{x - \bar x} + j(x) - j(\bar x) \ge 0
	\qquad\forall x \in X,
\end{equation*}
where $X, A, j$ are assumed to satisfy the conditions in \cref{asm:data}. Denote by $B_r^Z(z)$ the closed ball in a normed space $Z$ with radius $r>0$ and midpoint $z \in Z$. 
Suppose that a $p_0 \in P$, an $\bar x_0 \in S(p_0)$ and an $R>0$ are given such that $S(p_0) \cap B_R^X(\bar x_0) = \{\bar x_0\}$, 
such that $A$ is Fréchet-differentiable in $(p_0, \bar x_0)$ with partial derivatives $A_p \in \LL(P, Y)$ and $A_x \in \LL(X, Y)$, 
and such that the solution map $S$ is locally  nonempty and upper Lipschitzian at $p_0$ in the sense that
\begin{equation*}
 \emptyset \ne S(p) \cap B_R^X(\bar x_0)  \subset B_{L t}^X(\bar x_0)\quad \forall p \in B_{t}^P(p_0)
\end{equation*}
for some $L>0$ and all small enough $t>0$. Suppose further that the VI
	\begin{equation}
	\label{eq:VI_derivative_again}
		y \in X,\qquad \dual{A_p q + A_x y}{z - y}
		+ \frac12 Q\jxa(z)
		- \frac12 Q\jxa(y)
		\ge
		0
		\qquad\forall z \in X
	\end{equation}
with $a_0 := -A(p_0, \bar x_0)$ admits at most one solution for every $q \in P$ and assume that one of the following conditions 
is satisfied.
	\begin{enumerate}
		\item
			$j$ is weakly-$\star$ twice epi-differentiable in $\bar x_0$ for $a_0$ and $A_x$ is weakly-$\star$ completely continuous in the sense that
			 $y_n \weaklystar y$ in $X$ implies $A_x y_n \to A_x y$ in $Y$.\label{item:epi_scenario1}
			
		\item\label{item:tangible_2}
			$j$ is strongly twice epi-differentiable in $\bar x_0$ for $a_0$ and $A_x$ is such that $y_n \weaklystar y$ in $X$ implies $A_x y_n \weakly A_x y$ in $Y$ 
			and $\liminf_{n \to \infty } \dual{A_x y_n}{y_n} \geq \dual{A_x y}{y}$. \label{item:epi_scenario2}
			
	\end{enumerate}
 Then, \eqref{eq:VI_derivative_again} is uniquely solvable for all $q \in P$ and $S$ is weakly-$\star$ 
Hadamard directionally differentiable in the sense that for every family of parameters $\{q_t\}_{0 < t < t_0} \subset P$ that 
satisfies $q_t \to q$ for $t \searrow 0$ with some $q \in P$  and every family of solutions 
$\{\bar x_t\}_{0 < t < t_0} \subset X$ that satisfies $\bar x_t \in S(p_0 + t q_t) \cap B_R^X(\bar x_0)$ for all $0 < t < t_0$, it holds
\begin{equation}
\label{eq:randomconvergence42}
\frac{\bar x_t - \bar x_0}{t} \weaklystar y\quad \text{and}\quad \left \langle A_x \left ( \frac{\bar x_t - \bar x_0}{t}\right ),  \frac{\bar x_t - \bar x_0}{t} \right \rangle \to \dual{A_x y}{y}
\end{equation}
for $t \searrow 0$,
where $y$ is the unique solution to \eqref{eq:VI_derivative_again}. If, moreover, $z \mapsto \dual{A_x z}{z}$ is a Legendre form 
in the sense of \cite[Lemma 5.1b)]{ChristofWachsmuth2017:1}, then the convergence of the difference
quotients is even strong. 
\end{corollary}

\begin{proof}
If we start with a family of parameters $\{q_t\}_{0 < t < t_0} \subset P$ and a family of solutions $\{\bar x_t\}_{0 < t < t_0} \subset X$ 
as in the definition of the weak-$\star$ Hadamard 
directional differentiability, then we are precisely in the situation of \cref{assumption:sensitivity} 
(after translation by $p_0$) 
and \cref{thm:sufficient_abstract} immediately implies \eqref{eq:randomconvergence42}. To obtain the claim with the strong convergence, 
we just have to use the definition of the Legendre form. This completes the proof. 
\end{proof}

We remark that a special case of the above corollary may be found in \cite[Theorem~5.5]{BonnansShapiro2000}.

In practice, it is typically hard to check whether a given functional $j$ is twice epi-differentiable in a point $x \in  \dom(j)$ for some $g \in \partial j(x)$, cf., e.g., the calculations 
in \cite[Section 4]{ChristofMeyer2016} and \cite[Section 6.2]{ChristofWachsmuth2017:1}. The following lemma turns out to be helpful in this context not only in practical
applications but also for theoretical considerations.

\begin{lemma}[Criterion for Second-Order Epi-Differentiability]
	\label{lem:obtain_twice_epi}
	Let $x \in  \dom(j)$ and $g \in \partial j(x)$ be given.
	Suppose that there exist a set $Z \subset \KK\jxg$
	and a functional $Q : \KK\jxg \to [0,\infty)$
	such that 
	\begin{enumerate}
		\item
			for all $z \in \KK\jxg$ it holds $Q\jxg(z) \ge Q(z)$,
			\label{criterion:i}
		\item
			for all $z \in Z$
			and all $\{t_n\} \subset \R^+$ with $t_n \searrow 0$, there exists a sequence $\{z_n\}\subset X$ satisfying  $z_n \weaklystar z$, $\norm{z_n}_X \to \norm{z}_X$, 
			and \label{criterion:ii}
			\begin{equation*}
				Q(z)
				=
				\lim_{n \to \infty} \frac{j(x + t_n \, z_n) - j(x) - t_n \dual{g}{z_n}}{t_n^2/2}
				,
			\end{equation*}			
		\item
			for all $z \in \KK\jxg$ there exists a sequence $\{z_k\} \subset Z$ with $z_k \weaklystar z$, $\norm{z_k}_X \to \norm{z}_X$
			and
			$Q(z) \ge \liminf_{k \to \infty} Q(z_k)$.
			\label{criterion:iii}
	\end{enumerate}
	Then, $Q = Q\jxg$ and $j$ is strictly twice epi-differentiable in $x$ for $g$. 
	If, moreover, the sequences in \ref{criterion:ii} and \ref{criterion:iii} can be chosen to be strongly convergent, 
	then $j$ is even strongly twice epi-differentiable in $x$ for $g$. 
\end{lemma}
\begin{proof}
	We first prove the strict second-order epi-differentiability.
	From the properties of $Q$ and the definition of $Q\jxg$,
	we immediately obtain $Q = Q\jxg$ on $Z$.
	Assume now that a $z \in \KK\jxg$ and a sequence $\{t_n\} \subset \R^+$ with $t_n \searrow 0$ are given.
	Then, \ref{criterion:iii} implies that we can find a sequence 
	$\{z_k\} \subset Z$ with $\smash{z_k \weaklystar z}$, $\norm{z_k}_X \to \norm{z}_X$
	and
	$Q(z) \ge \liminf_{k \to \infty} Q(z_k)$.
	From \ref{criterion:ii}, we obtain further that for each  $z_k$
	there exists a sequence $\{z_{k,n}\}$ satisfying
	\begin{equation*}
		z_{k,n} \weaklystar z_k,\quad 
		\norm{z_{k,n}}_X \to \norm{z_k}_X,\quad\text{and}\quad
		 \frac{j(x + t_n \, z_{k,n}) - j(x) - t_n \dual{g}{z_{k,n}}}{t_n^2/2} \to Q(z_k)
	\end{equation*}
	for $n \to \infty$. Using exactly the same argumentation as in the proof of \cref{lem:more_properties_of_Q} \ref{item:more_properties_of_Q_2},
	we can now construct a sequence $\{\hat z_n\}$ with $\smash{\hat z_n \weaklystar z}$, $\norm{\hat z_n}_X \to \norm{z}_X$, and 
	\begin{equation*}
		 Q\jxg(z) \leq \liminf_{n \to \infty}\frac{j(x + t_n \, \hat z_{n}) - j(x) - t_n \dual{g}{\hat z_n}}{t_n^2/2} =  \liminf_{k \to \infty} Q(z_k) \leq Q(z) \leq Q\jxg(z),
	\end{equation*}
	where the first and the last estimate follow from \cref{def:weak_star_subderivative} and \ref{criterion:i}, respectively.  
	The above implies that $\{\hat z_n\}$ is a recovery sequence for $z$ as in the definition of the strict second-order epi-differentiability. 
	Since $z \in \KK\jxg$ was arbitrary, the first claim of the lemma now follows immediately, 
	cf.\  \cref{remark:secondorderepidifferentiability} \ref{remark:secondorderepidifferentiability:iii}. 

	To obtain the strong second-order epi-differentiability under the assumption of strong convergence in \ref{criterion:ii} and \ref{criterion:iii},
	we can proceed along exactly the same lines (just modify the selection argument in the proof of \cref{lem:more_properties_of_Q} \ref{item:more_properties_of_Q_2} accordingly).
\end{proof}

Using \cref{lem:obtain_twice_epi}, we obtain, e.g., the following result.

\begin{corollary}[Indicator Functions of Extended Polyhedric Sets]
\label{cor:extendedpolyhedric}
Let $K \subset X$ be a closed, convex, nonempty set, and denote by $\delta_K : X \to \{0, \infty\}$ the indicator function of $K$. Suppose that $X$ is reflexive, 
and assume that
an $x \in K$ and a $g \in \partial \delta_K (x)$ are given such that $K$ is extended polyhedric in $x$ for $g$ in the sense of 
\cite[Definition 3.52]{BonnansShapiro2000},
i.e., such that 
\begin{equation*}
\TT_K(x) \cap g^\perp = \mathrm{cl}\bigh(){ \left \{ z \in \TT_K(x)  \mid 0 \in \TT_K^{2}(x,z)  \right \}  \cap  g^\perp }
\end{equation*}
holds, where $\TT_K(x) := \mathrm{cl}(\R^+(K - x))$ and 
\begin{gather*}
\TT_K^{2}(x,z)
:=
\Bigh\{\}{
r \in X :  \dist\bigh(){ x + t \,z + {\textstyle\frac{1}{2}}\,t^2 \, r, K }
=
o(t^2) \text{ as } t \searrow 0
}
\end{gather*}
denote the tangent cone and the second-order tangent set at $x$ and $(x,z)$, respectively. Then, $\delta_K$ is strongly twice epi-differentiable in $x$ for $g$ and it holds
\begin{equation*}
\KK_{\delta_K}^{x,g}  = \TT_K(x) \cap g^\perp \quad \text{and}\quad  Q_{\delta_K}^{x,g}(z) = 0 \quad \forall z \in \KK_{\delta_K}^{x,g}.
\end{equation*}
\end{corollary}

\begin{proof}
We use \cref{lem:obtain_twice_epi} to prove the claim: Set $Z := \left \{ z \in \TT_K(x)  \mid 0 \in \TT_K^{2}(x,z)  \right \}  \cap  g^\perp$. 
Then, the definition of $\TT_K^{2}(x,z)$ implies that  for every $z \in Z$ and every $\{t_n\} \subset \R^+$ with $t_n \searrow 0$ 
there exists a sequence $\{r_n\} \subset X$ with $x + t_n z +\frac12 t_n^2 r_n \in K$ and $r_n \to 0$. 
The latter yields (cf.\ \cref{def:weak_star_subderivative})
\begin{equation*}
0 \leq Q_{\delta_K}^{x,g}(z)  \leq 
\liminf_{n \to \infty}  \left \langle - g,      r_n \right \rangle =0,
\end{equation*}
i.e., $Q_{\delta_K}^{x,g}(z) = 0$ for all $z \in Z$ and $Z \subset \KK_{\delta_K}^{x,g}$. 
From the definition of $Q_{\delta_K}^{x,g}$ and the lemma of Mazur, 
we obtain further that  $\KK_{\delta_K}^{x,g} \subset \TT_K(x) \cap g^\perp $. 
This implies in tandem with the extended polyhedricity of $K$ in $x$ for $g$ that the set $Z$ is dense in $\KK_{\delta_K}^{x,g}$. 
Defining $Q : \KK_{\delta_K}^{x,g}   \to [0, \infty)$, $Q(z) := 0$ for all $z \in \KK_{\delta_K}^{x,g}$, we now arrive exactly at the situation of 
\cref{lem:obtain_twice_epi} (with strong convergence). 
This allows us to deduce that $Q  = Q_{\delta_K}^{x,g} \equiv 0$ holds on $\KK_{\delta_K}^{x,g}$ and that $\delta_K$ is strongly twice epi-differentiable in $x$ for $g$. 
From \cref{lem:more_properties_of_Q} \ref{item:more_properties_of_Q_2}, we now obtain that $ Q_{\delta_K}^{x,g} : X \to [0, \infty]$ is 
weakly-$\star$ lower semicontinuous. 
This yields in combination with the density of $Z$ in 
$\TT_K(x) \cap g^\perp$ and the inclusion $\KK_{\delta_K}^{x,g} \subseteq \TT_K(x) \cap g^\perp$ that $\KK_{\delta_K}^{x,g} = \TT_K(x) \cap g^\perp$ 
and completes the proof. 
\end{proof}

Note that, in the situation of \cref{cor:extendedpolyhedric}, it is always true that
\begin{equation*}
\R^+(K - x) \subset  \left \{ z \in \TT_K(x)  \mid 0 \in \TT_K^{2}(x,z)  \right \},
\end{equation*}
cf.\ the definition of the second-order tangent set $\TT_K^{2}(x,z)$. 
This implies that every set that is polyhedric in the sense of \cite{Haraux1977} is also extended polyhedric in the sense of \cite{BonnansShapiro2000}.
In particular, we may combine \cref{corollary:tangible} with \cref{cor:extendedpolyhedric} to obtain a generalization of the classical differentiability result of
\cite{Mignot1976}, cf.\ also the results in \cite[Example 4.6]{Do1992} in this context.

We point out that there exist closed convex sets that satisfy the condition of extended polyhedricity but violate that of polyhedricity.
An easy example is the set 
\begin{equation*}
	\{0\} \cup \conv\bigh\{\}{ (1/n, 1/n^4) \in \R^2 \mid n \in \Z } \subset \R^2
	.
\end{equation*}

\section{Elliptic Variational Inequalities in Hilbert Spaces}
\label{sec:elliptic}

Having studied the very general setting of \cref{sec:abstract_analysis}, we now turn our attention to elliptic variational inequalities 
in Hilbert spaces. What is remarkable about VIs of this type is that the second-order epi-differentiability 
of the functional $j$ is not only sufficient for the directional differentiability of the solution map $S$  but also necessary. 
More precisely, we have the following result.

\begin{theorem}[Directional Differentiability for Elliptic Variational Inequalities]
\label{th:elliptic}
Let $X, A, j$ be as in \cref{asm:data}. Suppose that $X$ is a Hilbert space and assume that a $p_0 \in P$ and an $R>0$ are given 
such that the VI 
\begin{equation}
\label{eq:VI_again33}
	\bar x \in X,\qquad 
	\dual{A(p, \bar x)}{x - \bar x} + j(x) - j(\bar x) \ge 0
	\qquad\forall x \in X
\end{equation}
admits a unique solution $S(p) \in X$ for all $p \in B_R^P(p_0)$ and such that there exist constants $c,C >0$ with
\begin{equation}
\label{eq:monotonicityassumption}
 c \|x_1 - x_2\|_X^2 \leq \left \langle A(p_0, x_1) - A(p_0, x_2), x_1 - x_2 \right \rangle\quad \forall x_1, x_2 \in  X
\end{equation}
and 
\begin{equation}
\label{eq:lipschitzassumption}
\|A(p, x) - A(p_0, x)\|_{Y} \leq C \|p - p_0\|_P\quad \forall x \in X\quad \forall p \in B_R^P(p_0).
\end{equation}
Write $\bar x_0 := S(p_0)$ and $a_0 :=  -A(p_0, \bar x_0)$, and suppose that $A$ is Fréchet differentiable in $(p_0, \bar x_0)$ with partial derivatives $A_p \in \LL(P, Y)$ and $A_x \in \LL(X, Y)$.
Assume that $A_p$ is surjective. Then, the following statements are equivalent:
\setenumerate{label=(\Roman*)}
	\begin{enumerate}
		\item
			The solution map $S : B_R^P(p_0) \to X$ is strongly Hadamard directionally differentiable in $p_0$ in all directions $q \in P$. 
			\label{item:ddiff}
		\item
			The functional $j$ is strongly twice epi-differentiable in $\bar x_0$ for $a_0$.
			\label{item:ediff}
	\end{enumerate}
\setenumerate{label=(\roman*)}
Moreover, if one of these conditions is satisfied, then the following assertions hold.
\begin{enumerate}

\item $\smash{Q\jxa: X \to [0,\infty] }$ is proper, weakly lower semicontinuous and positively homogeneous of degree two, 
and  $\smash{\KK\jxa}$ is a pointed cone. \label{item:consequence1}

\item The directional derivative $y := S'(p_0;q) \in X$ in $p_0$ in a direction $q$ is uniquely characterized by the variational inequality \label{item:consequence2}
	\begin{equation}
	\label{eq:diffVIagain}
		y \in X,\qquad 
		\dual{A_p q + A_x y}{z - y}
		+ \frac12 Q\jxa(z)
		- \frac12 Q\jxa(y)
		\ge
		0
		\qquad\forall z \in X.
	\end{equation}
	Moreover,
	\begin{equation*}
		Q\jxa(y)
		=
		\lim_{t \searrow 0} \frac{j(\bar x_0 + t \, y_t) - j(\bar x_0) - t \, \dual{a_0}{y_t}}{t^2/2}
		,
	\end{equation*}
	where $y_t := (S(p_0 + t \, q) - \bar x_0 ) / t$,
	i.e.,
	the difference quotients $y_t$ are a recovery sequence for $y$.

\item  For every $z \in \smash{\KK\jxa}$ there exists a sequence $\{y_k\} \subset S'(p_0; P)$ with \label{item:consequence3}
\begin{gather*}
y_k \to z\quad\text{and}\quad  Q\jxa(y_k) \nearrow  Q\jxa(z)
\end{gather*}
as $k \to \infty$. In particular, $S'(p_0; P)$ is a dense subset of  $\smash{\KK\jxa}$.
\end{enumerate}
\end{theorem}

\begin{proof}
We first demonstrate that  part  \ref{item:epi_scenario2} of \cref{corollary:tangible} is applicable.
Set $\bar x_p := S(p)$ for all $p \in B_R^P(p_0)$. Then, \eqref{eq:VI_again33}, \eqref{eq:monotonicityassumption} and \eqref{eq:lipschitzassumption} yield
\begin{equation*}
\begin{aligned}
c \|\bar x_p - \bar x_0 \|_X^2   & \leq \left \langle A(p_0, \bar x_p) - A(p_0, \bar x_0), \bar x_p - \bar x_0 \right \rangle
\\
&\leq \left \langle A(p, \bar x_p) - A(p_0, \bar x_p), \bar x_0 - \bar x_p \right \rangle
\leq C \|p - p_0\|_{P} \|\bar x_p - \bar x_0 \|_X 
\end{aligned}
\end{equation*}
for all $p \in B_R^P(p_0)$, and we obtain
\begin{equation*}
\|\bar x_p - \bar x_0 \|_X  \leq \frac{C}{c} \| p - p_0\|_P\quad \forall p \in B_R^P(p_0). 
\end{equation*}
This shows that the solution map $S : B_R^P(p_0) \to X$ is Lipschitz at $p_0$ (in the classical sense).
From \eqref{eq:monotonicityassumption} and the Fréchet differentiability of $A$ in $(p_0, \bar x_0)$, it follows further
\begin{equation} 
c \|z\|_X^2 \leq \lim_{t \searrow 0} \frac{\left \langle A(p_0, \bar x_0 + tz) - A(p_0, \bar x_0), z \right \rangle}{t} = \left \langle A_x z, z\right \rangle 
\label{eq:derivative_ellipticity}
\end{equation}
for all $z \in X$, i.e.,  the bilinear form $(z_1, z_2) \mapsto \left \langle A_x z_1, z_2\right \rangle$ is elliptic.
Note that this ellipticity implies in particular that the map $z \mapsto \left \langle A_x z, z\right \rangle$ is a Legendre form. 
Consider now the VI \eqref{eq:diffVIagain} and assume that there exists a $q \in P$
such that \eqref{eq:diffVIagain} admits two solutions $y_1$ and $y_2$.
Then, it necessarily holds
\begin{align*}
		\dual{A_p q + A_x y_1}{y_2 - y_1}
		+ \frac12 Q\jxa(y_2)
		- \frac12 Q\jxa(y_1)
		&
		\ge
		0
		\qquad\text{and}
		\\
		\dual{A_p q + A_x y_2}{y_1 - y_2}
		+ \frac12 Q\jxa(y_1)
		- \frac12 Q\jxa(y_2)
		&\ge
		0,
\end{align*}
and we obtain by addition
\begin{equation}
\label{eq:uniqueness_estimate}
0 \geq \dual{ A_x y_1 - A_x y_2}{y_1 - y_2} \geq c \|y_1 - y_2\|_X^2.
\end{equation}
This shows that \eqref{eq:diffVIagain} can have at most one solution. 
If we combine all of the above, we see that the assumptions
of  part \ref{item:epi_scenario2} of \cref{corollary:tangible} are indeed satisfied
under condition \ref{item:ediff}.
The implication  \ref{item:ediff} $\Rightarrow$ \ref{item:ddiff} now follows immediately. 

Next, we check that \ref{item:consequence1} and \ref{item:consequence2}
hold under condition \ref{item:ediff}.
First, we note that
\cref{lem:basic_properties,lem:more_properties_of_Q}
yield that
\ref{item:ediff} implies \ref{item:consequence1}.
Further,
\cref{corollary:tangible} and \cref{thm:necessary_condition_1}
show that
\ref{item:ediff} implies \ref{item:consequence2}.

It thus only remains to prove that \ref{item:ddiff} $\Rightarrow$ \ref{item:ediff} holds 
and that one of the conditions \ref{item:ddiff} and  \ref{item:ediff} entails \ref{item:consequence3}.

So let us assume that \ref{item:ddiff} is satisfied, i.e., suppose that 
the map $S : B_R^P(p_0) \to X$ is strongly Hadamard directionally differentiable in $p_0$ in all directions $q \in P$.
Then, it follows from \cref{thm:necessary_condition_1} 
that the directional derivative $y := S'(p_0; q)$ of $S$ in $p_0$ in a direction $q \in P$ solves \eqref{eq:diffVIagain}. 
Since $S$ is directionally differentiable in $p_0$ in all directions $q \in P$ and since \eqref{eq:diffVIagain} can have at most one solution,
the latter implies that \eqref{eq:diffVIagain} possesses a unique solution $y$ for all $q \in P$ and that, if $y \in X$ solves \eqref{eq:diffVIagain}
with parameter $q$, then $y$ is necessarily identical to the directional derivative $S'(p_0; q)$. 
Consider now for $n \in \mathbb{N}_0$ and
\begin{equation*}
\varepsilon := \frac{c}{2 \|A_x\|_{\LL(X, Y)}} \in (0, \infty)
\end{equation*}
the variational inequality 
\begin{equation}
\label{eq:fixpoint_VI}
		y \in X,\qquad
		\dual{A_p q + (1 + n \varepsilon)A_x y}{z - y}
		+ \frac12 Q\jxa(z)
		- \frac12 Q\jxa(y)
		\ge
		0
		\qquad\forall z \in X.
\end{equation}
We claim that \eqref{eq:fixpoint_VI} admits a unique solution $y \in X$ for all $q \in P$ and all $n \in \mathbb{N}_0$.
Note that this unique solvability is indeed nontrivial since we do not know anything
about the second subderivative $Q\jxa$ at this moment. To show that \eqref{eq:fixpoint_VI} has a unique solution for all $q \in P$, we use induction on $n$. 
Since  \eqref{eq:fixpoint_VI} with $n=0$ is precisely \eqref{eq:diffVIagain}, the induction basis is trivial, 
so let us assume that the unique solvability is proved for some $n \in \mathbb{N}_0$. 
From the surjectivity of $A_p$, we obtain that for every $u \in X$ and every $q \in P$, there exists 
a $\tilde q \in P$ such that $A_p \tilde q = A_p q + \varepsilon A_x u$. The latter implies in combination with the
induction hypothesis that the VI
\begin{equation}
\label{eq:fixpoint_VI_intermediate}
		y \in X,\qquad
		\dual{A_p q + \varepsilon A_x u+ (1 + n \varepsilon)A_x y}{z - y}
		+ \frac12 Q\jxa(z)
		- \frac12 Q\jxa(y)
		\ge
		0
		\qquad\forall z \in X
\end{equation}
admits a unique solution for all $q \in P$ and all $u \in X$.
Fix $q$ and denote the solution operator $X \ni u \mapsto y \in X$ of \eqref{eq:fixpoint_VI_intermediate} by $T$.
Then, it follows analogously to the proof of \eqref{eq:uniqueness_estimate} that $T$ is globally Lipschitz with 
\begin{equation*}
 \|T(u_1) - T(u_2)\|_X  \leq \frac{\varepsilon}{(1 + n \varepsilon) c} \|A_x u_1 - A_x u_2\|_{Y} \leq \frac{1}{2} \|u_1 - u_2\|_X,
\end{equation*}
where the last estimate follows from the definition of $\varepsilon$. 
The above shows that $T$ is a contraction and implies, in combination with Banach's fixed-point theorem, 
that there exists a unique $u \in X$ with $Tu = u$, i.e., with
\begin{equation*}
		u \in X,\qquad
		\dual{A_p q +  (1 + (n+1) \varepsilon)A_x u}{z - u}
		+ \frac12 Q\jxa(z)
		- \frac12 Q\jxa(u)
		\ge
		0
		\qquad\forall z \in X.
\end{equation*}
This completes the induction step.

We are now in the position to prove \ref{item:ediff}. 
Consider an arbitrary but fixed $\tilde z \in \smash{\KK\jxa}$, 
and choose a sequence $\{q_n\} \subset P$ with $A_p q_n = -(1 + n \varepsilon )A_x \tilde z  $ for all $n \in \mathbb{N}$ 
(possible due to surjectivity).
Denote by $y_n$, $n \in \mathbb{N}$,
the unique solution to 
\begin{equation}
\label{eq:randomVI3623426}
		y_n \in X,\qquad
		\dual{A_p q_n + (1 + n \varepsilon)A_x y_n}{z - y_n}
		+ \frac12 Q\jxa(z)
		- \frac12 Q\jxa(y_n)
		\ge
		0
		\qquad\forall z \in X.
\end{equation}
Then, the choice $z = \tilde z$, the definition of $q_n$ and \eqref{eq:derivative_ellipticity} yield
\begin{equation*}
\begin{aligned}
c  (1 + n \varepsilon ) \|\tilde z - y_n\|_X^2  \leq (1 + n \varepsilon )\dual{A_x \tilde z - A_x y_n}{ \tilde z - y_n} 
\leq \frac12 Q\jxa(\tilde z) - \frac12 Q\jxa(y_n),
\end{aligned}
\end{equation*}
and we may deduce that
\begin{equation*}
Q\jxa(y_n) \leq Q\jxa(\tilde z) \qquad \text{and}\qquad
\|\tilde z - y_n\|_X^2 \leq \frac{1}{2c  (1 + n \varepsilon )}  Q\jxa(\tilde z)
\end{equation*}
holds for all $n \in \mathbb{N}$. 
Note that the surjectivity of $A_p$ implies that for every $y_n$ there exists a
$\tilde q_n \in P$ with $A_p \tilde q_n = A_p q_n + n \varepsilon A_x y_n$, 
and that for each such $\tilde q_n$ it necessarily holds
$y_n = S'(p_0; \tilde q_n)$ by \eqref{eq:randomVI3623426} and
\cref{thm:necessary_condition_1}.
We may thus conclude that for every $\tilde z \in \smash{\KK\jxa}$ we can find a sequence $y_n$ with
\begin{gather}
\label{eq:properties_y_n}
y_n  \in S'(p_0; P), \qquad y_n \to \tilde z\qquad\text{and}\qquad  \limsup_{n \to \infty} Q\jxa(y_n) \leq  Q\jxa(\tilde z).
\end{gather}
Using that for each $y \in S'(p_0; P)$ there exists a recovery sequence as in the definition of the strong second-order epi-differentiability, 
see \eqref{eq:recovery_sequence},   
and applying \cref{lem:obtain_twice_epi} with $Q = Q\jxa$ and $Z = S'(p_0; P)$, it now follows straightforwardly that
$j$ is strongly twice epi-differentiable in $\bar x_0$ for $a_0$. This shows that \ref{item:ddiff} indeed implies \ref{item:ediff}. 

To finally prove \ref{item:consequence3}, we note 
 that the strong second-order epi-differentiability of $j$ in $\bar x_0$ for $a_0$ yields the weak lower semi-continuity of 
the functional
 $Q\jxa$, see \cref{lem:more_properties_of_Q}\ref{item:more_properties_of_Q_2}. 
This allows us to continue the last estimate in \eqref{eq:properties_y_n}
as follows
\begin{equation*}
Q\jxa(\tilde z) \geq \limsup_{n \to \infty} Q\jxa(y_n) \geq \liminf_{n \to \infty} Q\jxa(y_n) \geq Q\jxa(\tilde z),
\end{equation*}
i.e., $Q\jxa(y_n) \to Q\jxa(\tilde z)$ as $n \to \infty$. Since $Q\jxa(y_n) \leq Q\jxa(\tilde z)$ holds by the construction of $y_n$,
\ref{item:consequence3} follows immediately. This completes the proof. 
\end{proof}

Some remarks regarding \cref{th:elliptic} are in order.

\begin{remark}\hfill
\label{asm:Hilbert_data}
\begin{enumerate}
\item 
The VIs in \cref{th:elliptic} are, in fact, not elliptic variational inequalities in the classical sense since
we do not assume, e.g., that $j$ is convex and lower semicontinuous. 
The classical setting, i.e., the situation where $P = X\dualspace$ and 
where \eqref{eq:VI_again33}
takes the form 
\begin{equation*}
	\bar x \in X,\qquad 
	\dual{A(\bar x)}{x - \bar x} + j(x) - j(\bar x) \ge \dual{p}{ x - \bar x}
	\qquad\forall x \in X
\end{equation*}
with a strongly monotone and Fréchet differentiable operator $A$ and a convex, lower semicontinuous and proper functional $j$ (cf.\  \cite[Section 1.1]{Adly2017}),
is, of course, covered by  \cref{th:elliptic} as one may easily check. 
Note that, for a classical elliptic variational inequality, the uniqueness and existence of solutions $S(p)$, $p \in X\dualspace$,
follows immediately from the theory of pseudomonotone operators, cf.\ \cite[Section 1.7]{KikuchiOden1980}.

\item We point out that \cref{th:elliptic} significantly generalizes \cite[Theorem~4.3]{Do1992}, where the equivalence of
\ref{item:ddiff} and \ref{item:ediff} is proved under the assumption that $P=X$, 
that $\dual{A(p, x_1)}{x_2} = (x_1 - p, x_2)_X$ for all $x_1, x_2, p \in X$ (where $(.,.)_X$ denotes the inner product in $X$), and
that $j$ is convex, proper and lower semicontinuous (see also \cite[Proposition 6.3]{BorweinNoll1994} in this context). 
Note that we have proved \cref{th:elliptic} without ever
using the concept of protodifferentiability, and that \cref{th:elliptic}
also covers those cases where the VI \eqref{eq:VI_again33} cannot be identified with a minimization problem of the form \eqref{eq:minimization_problem}
and where, as a consequence, the method of proof in  \cite{Do1992} cannot be applied.

\item Although the proof of \cref{th:elliptic} does not need the Hilbert space structure of $X$,
it is not useful to assume that $X$ is only a Banach space in the situation of \cref{th:elliptic} since
the existence of a Fréchet differentiable map $A$ with the property \eqref{eq:monotonicityassumption}  already implies that
$X$ is Hilbertizable, see \eqref{eq:derivative_ellipticity}.

\end{enumerate}
\end{remark}

\section{Three Applications}
\label{sec:examples}
In what follows, we demonstrate by means of three tangible examples that 
the results in \cref{sec:abstract_analysis,sec:tangible_corollary,sec:elliptic} 
are not only interesting from a theoretical point of view but also suitable for practical applications
which are not covered by the existing literature.

\subsection{Static Elastoplasticity in Dual Formulation}
\label{subsec:plasti}
In this section,
we demonstrate that \cref{corollary:tangible}
enables us to differentiate the solution map
of static elastoplasticity.
In \cite{HerzogMeyerWachsmuth2010:2},
it was shown that this map is weakly directionally differentiable.
This result was sharpened in \cite[Theorem~3.8]{BetzMeyer2015}
to Bouligand differentiability under more restrictive regularity assumptions.
Our technique enables us to prove Hadamard directional differentiability
under the natural regularity of the problem.

We will work in a slightly abstract setting.
However, we roughly keep the notation of
\cite{HerzogMeyerWachsmuth2010:2,BetzMeyer2015}
to make it easier to transfer our results to
the precise setting of static elastoplasticity.

\begin{assumption}[Setting of Elastoplasticity]
	\label{asm:plasti}
	We assume that $V$ is a Hilbert space, that $\mu$
	is a $\sigma$-finite measure on a set $\Omega$, and that $m, n \in \mathbb{N}$.
	We further suppose that a bounded, linear, symmetric and coercive
	map $A : S^2 \to (S^2)\dualspace$, a bounded linear 
	map $B : S^2 \to V\dualspace$ and a linear operator 
	$\DD : \R^m \to \R^n$
	are given, where $S^2 := L^2(\mu)^m$. 
	Via 
	$\DD$
	we define
	the set
	$K := \{ \Sigma \in S^2 : \abs{\DD\Sigma}_{\R^n} \le 1\text{ a.e.\ in }\Omega\}$.
	Finally, we assume that the restriction of $B$ to the Hilbert space $H:=\{\Sigma \in S^2 :  \abs{\DD\Sigma}_{\R^n} = 0 \text{ a.e.\ in }\Omega\}$
	is surjective.
\end{assumption}
Although we do not need the product-type structure of $S^2$,
we keep this notation for consistency with the above references.
For the same reason, we 
use the symbol $A$ to denote the first operator appearing in \cref{asm:plasti}.
(This operator will only be a part of the nonlinearity in \eqref{eq:VI} and should 
not be confused with it, cf.\ \eqref{eq:operatorAelastoplasti} below).

By $\TT_K(\Sigma), \NN_K(\Sigma) \subset S^2$
we denote the tangent cone and the normal cone to $K$ at $\Sigma \in K$
in the sense of convex analysis, respectively.

Within the above framework, we consider
for a given datum $\ell \in V\dualspace$ the minimization problem
\begin{equation}
	\label{eq:plasti_1}
	\text{Minimize}\quad
	\frac12 \, \dual{A \, \Sigma}{\Sigma}
	\quad\text{such that}\quad
	B\Sigma = \ell
	\quad\text{and}\quad
	\Sigma \in K.
\end{equation}
First, we provide a result concerning the solvability of \eqref{eq:plasti_1}.
\begin{lemma}
\label{lemma:plastiproperties}
Problem \eqref{eq:plasti_1} admits a unique solution $\Sigma_\ell$ for every $\ell \in V\dualspace$.
Moreover, for every $\ell \in V\dualspace$ with associated solution $\Sigma_\ell$, there exists a unique multiplier
$u_\ell \in V$ with
\begin{equation}
	\label{eq:plasti_VI}
	\dual{A \Sigma_\ell + B\dualspace u_\ell}{T - \Sigma_\ell}
	- \dual{B \Sigma_\ell - \ell}{v - u_\ell}
	\ge
	0
	\qquad
	\forall (T,v) \in K \times V,
\end{equation}
and the solution map $V\dualspace \ni \ell \mapsto (\Sigma_\ell, u_\ell) \in S^2 \times V$ is globally Lipschitz continuous. 
\end{lemma}
\begin{proof}
By our assumptions, the feasible set of \eqref{eq:plasti_1} is closed, convex and nonempty. 
Together with the continuity, radial unboundedness and strict convexity of the objective,
this yields the existence of a unique solution $\Sigma_\ell$ for all $\ell \in V \dualspace$. 
Further, the CQ
of Zowe and Kurcyusz
is satisfied by \eqref{eq:plasti_1}.
Thus, there exist $u_\ell \in V$
and $\Xi_\ell \in \NN_K(\Sigma_\ell)$
with
\begin{equation*}
	A \Sigma_\ell + B\adjoint u_\ell + \Xi_\ell = 0.
\end{equation*}
The above and the equality $B\Sigma_\ell = \ell$ immediately give \eqref{eq:plasti_VI}. 
Next, we prove the Lipschitz continuity of the solution map.
By using that $B : H \to V\dualspace$ is surjective and that $H$ is a Hilbert space,
we find that $B$ admits a bounded linear right inverse $\tilde B: V\dualspace \to H$, i.e., $B \tilde B = \mathrm{Id}_{V\dualspace}$.
Consider now two right-hand sides $\ell_1, \ell_2 \in V\dualspace$ with associated solutions $\Sigma_{\ell_1}, \Sigma_{\ell_2}$
and multipliers $u_{\ell_1}, u_{\ell_2}$. Then, we may choose the tuple $(T, v) := (\Sigma_{\ell_2} + \tilde B(\ell_1 - \ell_2), u_{\ell_1})$
in the VI for $\ell_1$ and the tuple $(T, v) := (\Sigma_{\ell_1} + \tilde B(\ell_2 - \ell_1), u_{\ell_2})$ in the VI for $\ell_2$ to obtain
\begin{gather*}
	\dual{A \Sigma_{\ell_1}}{\Sigma_{\ell_2} - \Sigma_{\ell_1} + \tilde B(\ell_1 - \ell_2) }
	\ge
	0
	\quad
	\text{and}
	\quad
	\dual{A \Sigma_{\ell_2}}{\Sigma_{\ell_1} - \Sigma_{\ell_2} + \tilde B(\ell_2 - \ell_1) }
	\ge
	0.
\end{gather*}
If we add the above, it follows immediately that the solution map to \eqref{eq:plasti_1} is  Lipschitz in the 
$\Sigma$-component. If, on the other hand, we choose tuples of the form $(\Sigma_{\ell_1} + \tilde B \hat\ell , u_{\ell_1})$
and $(\Sigma_{\ell_2} + \tilde B \hat\ell, u_{\ell_2})$, $\hat\ell \in V\dualspace$, in the VIs for 
$\ell_1$ and $\ell_2$, respectively, then we obtain 
\begin{gather*}
	\dual{A \Sigma_{\ell_1} + B\dualspace u_{\ell_1}}{\tilde B \hat\ell }
	\ge
	0
	\quad
	\text{and}
	\quad
	\dual{A \Sigma_{\ell_2} + B\dualspace u_{\ell_2}}{\tilde B \hat\ell }
	\ge
	0
	\qquad\forall \hat\ell \in V\dualspace
\end{gather*}
and, consequently,
\begin{equation}
\label{eq:diffbootstrapping}
\dual{u_{\ell_1} - u_{\ell_2}}{\hat\ell} = \dual{A \Sigma_{\ell_2}  - A \Sigma_{\ell_1}}{\tilde B \hat\ell }\qquad \forall \hat\ell \in V\dualspace.
\end{equation}
The above yields that the multiplier $u_\ell$ is unique for each $\ell$ and that the map $\ell \mapsto u_\ell$ is globally Lipschitz continuous, too. 
\end{proof}

We emphasize that the linear operator which defines the VI \eqref{eq:plasti_VI}
has saddle-point structure.
Thus, it is not coercive and the VI cannot be identified with a projection problem.
In particular, we cannot apply the classical results of, e.g.,  \cite{Do1992}
to obtain the directional differentiability of the solution operator.

As a preparation for our differentiability result,
we give an expression for the second subderivative
of the indicator function of $K$.
\begin{lemma}
	\label{lem:subderivative_indicator}
	Let $\Sigma \in K$ and $\Xi \in \NN_K(\Sigma)$ be given.
	Then, there exists a $\lambda \in L^2(\mu)$
	with $\lambda \ge 0$
	such that $\Xi = \lambda \, \DD\adjoint \DD \Sigma$.
	Moreover,
	the indicator function $\delta_K$
	is strongly twice epi-differentiable in $\Sigma$ for $\Xi$
	with
	\begin{equation*}
		Q_{\delta_K}^{\Sigma,\Xi}
		(T)
		=
		\int_\Omega
		\lambda \, \abs{\DD T}_{\R^n}^2
		\,
		\d\mu
		\in [0,\infty]
	\end{equation*}
	for all $T \in \TT_K(\Sigma) \cap \Xi\anni$
	and $Q_{\delta_K}^{\Sigma,\Xi}(T) = +\infty$ otherwise.
	In particular,
	\begin{equation*}
		\KK_{\delta_K}^{\Sigma, \Xi}
		=
		\Bigh\{\}{
			T \in \TT_K(\Sigma) \cap \Xi\anni
			\mid
			\int_\Omega
			\lambda \, \abs{\DD T}_{\R^n}^2
			\,
			\d\mu
			<
			\infty
		}.
	\end{equation*}
\end{lemma}
\begin{proof}
	It is easy to check that
	$\Xi(x)$ is a nonnegative multiple of $\DD\adjoint\DD\Sigma(x)$
	for a.e.\ $x \in \Omega$ with $\abs{\DD\Sigma(x)}_{\R^n} = 1$,
	and zero otherwise.
	Moreover, it is easy to see that there exists a constant $c > 0$ with 
	$|\DD\dualspace \DD z|_{\R^m} \geq c | \DD z|_{\R^n}$
	for all $z \in \R^m$. 
	Combining these facts, we obtain that the function 
	\begin{equation*}
	\lambda(x)
	:=
	\begin{cases}
	\scalarprod{\Xi(x)}{\DD\adjoint\DD\Sigma(x)}_{\R^m} / \abs{\DD\adjoint\DD\Sigma(x)}_{\R^m}^2 & \text{ if }  \abs{\DD\Sigma(x)}_{\R^m} = 1
	\\
	0 & \text{ else }
	\end{cases}
	\end{equation*}
	satisfies $0 \leq \lambda \in L^2(\mu)$ and $\Xi = \lambda \, \DD\adjoint \DD \Sigma$ as claimed.
	The formula
	for the second subderivative
	follows from
	\cite[Section~5]{Do1992} and \cite[Exercise~13.17]{RockafellarWets1998}.
	In particular, \cite[Theorem~5.5]{Do1992}
	implies that $\delta_K$
	is strongly twice epi-differentiable.
\end{proof}

We are now in the position to prove the main result of this section.
\begin{theorem}
	\label{thm:plasti_diff}
	Let $\ell_0 \in V\dualspace$ be given.
	Denote by $\lambda_0 \in L^2(\mu)$ the function from \cref{lem:subderivative_indicator}
	such that
	\begin{equation*}
		A \Sigma_{\ell_0} + B\adjoint u_{\ell_0} + \lambda_0 \, \DD\adjoint \DD\Sigma_{\ell_0} = 0.
	\end{equation*}
	Then,
	the mapping $V\dualspace \ni \ell \mapsto (\Sigma_\ell, u_\ell) \in S^2 \times V$
	is strongly Hadamard directionally differentiable in $\ell_0$.
	Moreover, 
	the directional derivative $(\Sigma', u') \in \KK_{\delta_K}^{\Sigma_{\ell_0},\Xi_{\ell_0}} \times V$ in direction $\delta\ell \in V\dualspace$
	is given by the unique solution of
	\begin{equation}
		\label{eq:plasti_deriv_VI_2}
		\begin{aligned}
			\dual{A \Sigma' + B\adjoint u'}{T - \Sigma'}
			+
			\int_\Omega
			\lambda_0 \, \bigscalarprod{\DD\Sigma'}{\DD(T - \Sigma')}_{\R^n}
			\,
			\d\mu
			- \dual{B\Sigma' - \delta\ell}{v - u'}
			\ge
			0
			\quad
			&
			\\
			\forall (T,v) \in \KK_{\delta_K}^{\Sigma_{\ell_0},\Xi_{\ell_0}}\times V.
			&
		\end{aligned}
	\end{equation}
\end{theorem}
\begin{proof}
	In what follows, our aim is to apply 
	\cref{corollary:tangible}
	under its assumption \ref{item:tangible_2}.
	To this end, we note that, if we set $X := S^2 \times V$, $Y := S^2 \times V\dualspace$,
	$P := V\dualspace$ and
	\begin{equation}
	\label{eq:operatorAelastoplasti}
		\AA : P \times X \to Y,\qquad 
		\AA( \ell, (\Sigma, u) )
		:=
		(A \Sigma + B\adjoint u, -B \Sigma + \ell ),
	\end{equation}
	and if we define $j$ to be the indicator function of $K \times V$,
	then  \eqref{eq:plasti_VI} takes exactly the form \eqref{eq:VI}. 
	In this setting, the linearized VI \eqref{eq:VI_derivative_again} becomes
	\begin{equation}
		\label{eq:plasti_deriv_VI}
		\begin{aligned}
		\dual{A \Sigma' + B\adjoint u'}{T - \Sigma'}
		+
		\int_\Omega
		\frac{\lambda_0}{2} \, \bigh[]{ \abs{\DD T}_{\R^n}^2 - \abs{\DD \Sigma'}_{\R^n}^2 }
		\,
		\d\mu
		- \dual{B  \Sigma' - \delta\ell}{v - u'}
		\ge
		0
		\quad
		&\\
		\forall (T,v) \in \KK_{\delta_K}^{\Sigma_{\ell_0},\Xi_{\ell_0}} \times V.
		&
		\end{aligned}
	\end{equation}
	Using exactly the same arguments as in the proof of \cref{lemma:plastiproperties}, we obtain that 
	\eqref{eq:plasti_deriv_VI} can have at most one solution.
	Finally, \ref{item:tangible_2} in \cref{corollary:tangible} follows from
	\cref{lem:subderivative_indicator}
	and the fact that
	$(T,v) \mapsto \dual{\AA_x(T,v)}{(T,v)} = \dual{AT}{T}$
	is convex and continuous, thus weakly lower semicontinuous.
	\cref{corollary:tangible} now
	yields that the solution map is weakly 
	Hadamard directionally differentiable and that 
	the directional derivatives are uniquely characterized by \eqref{eq:plasti_deriv_VI}.
	Moreover, \eqref{eq:randomconvergence42}
	together with
	$\dual{\AA_x(T,v)}{(T,v)} = \dual{AT}{T}$
	implies that the difference quotients associated with the $S^2$-component
	converge strongly.
	The strong convergence of the difference quotients associated with the $V$-component now
	follows from \eqref{eq:diffbootstrapping}.
	Finally, the equivalence of \eqref{eq:plasti_deriv_VI} and \eqref{eq:plasti_deriv_VI_2}
	is easy to check.
\end{proof}

\subsection{Projection onto a Prox-Regular Set}
\label{subsec:prox_regular}
Next, we show that \cref{th:elliptic}
can also be used to study  the differentiability properties
of the projection onto a prox-regular set.
The notion of prox-regular sets generalizes the concept of convexity and has been introduced
many times with different names. Some unification was
performed in \cite{PoliquinRockafellarThibault2000,ColomboThibault2010}.
Throughout this section, $K$ denotes a  closed nonempty subset of a real Hilbert space $H$.

\pagebreak[3]

\begin{definition}[Prox-Regularity]
	\label{def:prox_reg}
~	\begin{enumerate}
		\item
			The set-valued projection of $y \in H$ onto $K$ is given by
			\begin{equation*}
				\pi_K(y)
				:=
				\Bigh\{\}{
					x \in K
					\mid
					\norm{x - y}_H = \inf_{x' \in K} \norm{x' - y}_H
				}.
			\end{equation*}
		\item
			For $x \in K$ the proximal normal cone is given by
			\begin{equation*}
				\NN_K^P(x)
				:=
				\bigh\{\}{
					v \in H
					\mid
					\exists \lambda \ge 0, y \in H:
					x \in \pi_K(y)
					\mspace{9mu}\text{and}\mspace{9mu}
					v = \lambda \, (y - x)
				}.
			\end{equation*}
		\item
			For $r > 0$, the set $K$ is called
			$r$-prox-regular, if
			\begin{equation}
				\label{eq:r_prox_reg}
				\frac12 \, \norm{v}_H \, \norm{x - \bar x}_H^2
				\ge
				r \,
				\scalarprod{v}{x - \bar x}_H
				\qquad
				\forall \bar x,x \in K, v \in \NN_K^P(\bar x).
			\end{equation}
	\end{enumerate}
\end{definition}
Note that the set $K$ is convex if and only if $K$ is $r$-prox-regular for all $r > 0$.
We give a geometric interpretation of \eqref{eq:r_prox_reg}.
For a given $\bar x \in K$ and a $v \in \NN_K^P(\bar x)$ with $\norm{v}_H < r$, this condition implies
\begin{equation}
	\label{eq:inequality_prox_regular}
	\begin{aligned}
		\bignorm{ x - (\bar x + v)}_H^2
		&=
		\Bigh(){\frac{\norm{v}_H} r + 1 - \frac{\norm{v}_H} r}\,\norm{ x - \bar x}_H^2 + \norm{v}_H^2 - 2 \, \scalarprod{v}{x - \bar x}_H
		\\
		&\ge
		\Bigh(){1 - \frac{\norm{v}_H} r}\,\norm{ x - \bar x}_H^2 + \norm{v}_H^2
		\ge
		\norm{v}_H^2\qquad \forall x \in K,
	\end{aligned}
\end{equation}
where the last inequality is strict for $x \neq \bar x$.
Hence, the intersection of the closed ball $B_{\norm{v}_H}^H(\bar x + v)$ with $K$
is precisely the point $\bar x$.
In particular, $\pi_K(\bar x + v) = \{\bar x\}$.

The above statement can be strengthened as follows,
see \cite[Theorem~4.1, Lemma~4.2]{PoliquinRockafellarThibault2000},
\cite[Theorem~0.16]{ColomboThibault2010}.
\begin{theorem}
	\label{thm:lipschitz_projection}
	Assume that the closed set $K \subset H$ is $r$-prox-regular for some $r > 0$.
	We fix $\rho \in (0,r)$ and consider the $\rho$-enlargement
	\begin{equation*}
		K_\rho
		:=
		\{ y \in H \mid \exists x \in K : \norm{x - y}_H < \rho\}.
	\end{equation*}
	Then, $\pi_K$ is single-valued on $K_\rho$
	and we define $\proj_K(y)$ to be the single element in $\pi_K(y)$
	for all $y \in K_\rho$.
	Moreover,
	$\proj_K$ is Lipschitz continuous on $K_\rho$
	with rank $r / (r - \rho)$.
\end{theorem}
For later use, we recast \eqref{eq:inequality_prox_regular}
in the setting of \cref{thm:lipschitz_projection}.
Let us fix $p \in K_\rho$, $\rho \in (0, r)$.
We set $\bar x = \proj_K(p)$ and $v = p - \bar x \in \NN_K^P(\bar x)$.
Applying \eqref{eq:inequality_prox_regular} yields
\begin{equation}
	\label{eq:inequality_prox_regular_again}
	\begin{aligned}
		\norm{x - p}_H^2
		&\ge
		\Bigh(){1 - \frac{\norm{p - \bar x}_H} r}\,\norm{ x - \bar x}_H^2 + \norm{p - \bar x}_H^2
		\\
		&\ge
		\Bigh(){1 - \frac\rho r}\,\norm{ x - \bar x}_H^2 + \norm{p - \bar x}_H^2
		\qquad \forall x \in K
		.
	\end{aligned}
\end{equation}
In what follows, we first link the second-order epi-differentiability of $\delta_K$
to the differentiability of an auxiliary function.
\begin{lemma}
	\label{lem:epi_diff_j}
	Assume that the closed set $K \subset H$ is $r$-prox-regular for some $r > 0$.
	We fix $p_0 \in K_r$
	and set $\bar x = \proj_K(p_0)$.
	If $\delta_K$ is strongly twice epi-differentiable
	in $\bar x$ for $p_0 - \bar x$,
	then
	the function $j := \frac12 \, \norm{\cdot}_H^2 + \delta_K$ is strongly twice epi-differentiable
	in $\bar x$ for $p_0$.
\end{lemma}

\begin{proof}
	By invoking the definition of $Q_j^{\bar x,p_0}$, for every $z \in H$, we find
	\begin{align*}
		&Q_j^{\bar x,p_0}(z)
		\\
		&=
		\inf
		\biggh\{\}{
			\liminf_{n \to \infty} \frac{\delta_K(\bar x + t_n \, z_n)  - \delta_K(\bar x) 
			 - t_n \, \scalarprod{p_0 - \bar x}{z_n}_H}{t_n^2/2}  +   \norm{ z_n}_H^2
			\mid
			t_n \searrow 0,
			z_n \weakly z
		}
		\\
		&\ge  \norm{ z}_H^2 +
		\inf
		\biggh\{\}{
			\liminf_{n \to \infty} \frac{\delta_K(\bar x + t_n \, z_n)  - \delta_K(\bar x) 
			 - t_n \, \scalarprod{p_0 - \bar x}{z_n}_H}{t_n^2/2}    
			\mid
			t_n \searrow 0,
			z_n \weakly  z
		}
		\\
		&= 
		\norm{z}_H^2 + Q_{\delta_K}^{\bar x, p_0-\bar x}(z).
	\end{align*}
	On the other hand, by the strong second-order epi-differentiability of $\delta_K$ in $\bar x$ for $p_0 - \bar x$,
	given $t_n \searrow 0$, we find for every $z \in H$ a sequence $z_n$ with $z_n \to z$ and
	\begin{align*}
		\norm{z}_H^2 + Q_{\delta_K}^{\bar x,p_0-\bar x}(z)
		&=
		\norm{z}_H^2 + \lim_{n \to \infty} \frac{\delta_K(\bar x + t_n \, z_n)  - \delta_K(\bar x) 
			 - t_n \, \scalarprod{p_0 - \bar x}{z_n}_H}{t_n^2/2}
		\\
		&=
		\lim_{n \to \infty}\frac{j(\bar x + t_n \, z_n) - j(\bar x) - t_n \, \scalarprod{p_0}{z_n}_H}{t_n^2/2}
		\ge
		Q_j^{\bar x,p_0}(z).
	\end{align*}
	These two inequalities
	show that
	$j$ is strongly twice epi-differentiable at $\bar x$ for $p_0$
	with
	\begin{equation}
		\label{eq:epi_umrechnen}
		Q_j^{\bar x,p_0}(z)
		=
		\norm{z}_H^2 + Q_{\delta_K}^{\bar x,p_0-\bar x}(z)
		\qquad\forall z \in H.
	\end{equation}
\end{proof}

Now, we present the main theorem of this section.
\begin{theorem}
	\label{thm:diff_prox_reg}
	Assume that the closed set $K \subset H$ is $r$-prox-regular for some $r > 0$.
	We fix $p_0 \in K_r$
	and set $\bar x := \proj_K(p_0)$.
	Then
	the following assertions are equivalent.
	\begin{enumerate}
		\item
			\label{item:prox_reg_2}
			For all $\rho \in (\norm{p_0 - \bar x}_H , r)$,
			the function $j := \frac12 \, \norm{\cdot}_H^2 + \delta_K$ is strongly twice epi-differentiable
			in $\bar x$ for $p_0 + (1-\frac\rho r) (p_0-\bar x)$.
		\item
			\label{item:prox_reg_3}
			For some $\rho \in (\norm{p_0 - \bar x}_H , r)$,
			the function $j$ is strongly twice epi-differentiable
			in $\bar x$ for $p_0 + (1-\frac\rho r) (p_0-\bar x)$.
		\item
			\label{item:prox_reg_4}
			The projection $\proj_K$ is Hadamard directionally differentiable at $p_0$.
	\end{enumerate}
\end{theorem}
\begin{proof}
	We show that for each fixed $\rho \in (\norm{p_0 - \bar x}_H, r)$,
	assertion
	\ref{item:prox_reg_4}
	is equivalent to the strong second-order epi-differentiability of $j$
	in $\bar x$ for $p_0 + (1-\frac\rho r) (p_0-\bar x)$.

	We are going to apply \cref{th:elliptic}
	with the setting $X = Y = P = H$,
	\begin{equation*}
		A(p,x) = \bigh(){1 - \frac\rho r}(x - p) - p,
		\quad
		j(x) = \frac12\,\norm{x}_H^2 + \delta_K(x)
		.
	\end{equation*}
	We first check that the VI \eqref{eq:VI_again33} 
	has a unique solution for all $p \in K_\rho$ and that this solution is exactly the projection of $p$ onto $K$.
	To this end, we set $\bar x_p := \proj_K(p)$, and note that \eqref{eq:r_prox_reg} and \eqref{eq:inequality_prox_regular_again}
	yield
	\begin{align*}
		\dual{A(p,\bar x_p)}{x - \bar x_p} + j(x) - j(\bar x_p)
		&=
		\bigh(){1 - \frac\rho r} \, \scalarprod{\bar x_p - p}{x - \bar x_p}_H
		+ \frac{\norm{x - p}_H^2}{2} - \frac{\norm{\bar x_p - p}_H^2}{2}
		\\
		&\ge
		- \bigh(){1 - \frac\rho r} \, \frac{\norm{\bar x_p - p}_H}{2 \, r} \norm{x - \bar x_p}_H^2
		+\bigh(){1 - \frac\rho r} \,  \frac{\norm{x - \bar x_p}_H^2}{2}
		\\
		&\ge 0\qquad \forall x \in K.
	\end{align*}
	By the ellipticity of $A$, the solution of \eqref{eq:VI_again33} is also unique.
	Hence, $\bar x_p = \proj_K(p)$ is the unique solution of \eqref{eq:VI_again33} with parameter $p$
	for all $p \in K_\rho$. 
	The equivalence of 
	\ref{item:prox_reg_3} and \ref{item:prox_reg_4}
	now follows immediately from \cref{th:elliptic}.
\end{proof}
From the last theorem, we easily get the following corollary,
which is proven in
\cite[Proposition~2.2]{Noll1995} in the case that $K$ is convex.
\begin{corollary}
	\label{cor:diff_segments}
	Assume that the closed set $K \subset H$ is $r$-prox-regular for some $r > 0$.
	We fix $p_0 \in K_r$
	and set $\bar x := \proj_K(p_0)$, $v = (p_0 - \bar x)/\norm{p_0 - \bar x}_H$.
	Then
	the following assertions are equivalent.
	\begin{enumerate}
		\item
			\label{item:segment_1}
			The projection $\proj_K$ is Hadamard differentiable
			in $\bar x + \rho \, v$
			for one $\rho \in (0,r)$.
		\item
			\label{item:segment_2}
			The projection $\proj_K$ is Hadamard differentiable
			in $\bar x + \rho \, v$
			for all $\rho \in (0,r)$.
		\item
			\label{item:segment_3}
			For one $\tilde\rho \in (0,r)$,
			the function $j$ is strongly twice epi-differentiable
			in $\bar x$ for $\bar x + \tilde\rho \, v$.
		\item
			\label{item:segment_4}
			For all $\tilde\rho \in (0,r)$,
			the function $j$ is strongly twice epi-differentiable
			in $\bar x$ for $\bar x + \tilde\rho \, v$.
	\end{enumerate}
\end{corollary}
\begin{proof}
	Let us fix $\rho \in (0,r)$
	and $\tilde\rho \in \bigh(){\rho, (2-\frac\rho r)\,\rho}$.
	Then,
	\cref{thm:diff_prox_reg} shows that
	\ref{item:segment_1}
	holds for $\rho$
	if and only if
	\ref{item:segment_3}
	holds for $\tilde\rho$.

	Now, any two points in the set
	$\bigh\{\}{(\rho, \tilde\rho) \in (0,r)^2 \mid \rho < \tilde\rho < (2-\frac\rho r)\,\rho}$
	can be connected by a finite polygonal path
	whose edges are parallel to the coordinate axes.
	This yields the claim.
\end{proof}

Before we conclude this section, 
we remark that the converse of \cref{lem:epi_diff_j}
can be easily established if the space $H$ is finite dimensional.
In infinite dimensions, however, this reverse implication does not hold
anymore in general. 
Consider, for example, the set
$K = \{x \in H \mid \norm{x}_H \ge 1\}$
in an infinite-dimensional Hilbert space $H$
with orthonormal system $\{e_i\}_{i\in\N}$.
It is easy to check that $K$ is $1$-prox-regular
and that the projection onto $K$ is directionally differentiable
on $H \setminus \{0\}$.
Define $p_0:= \frac12 e_1$ and $\bar x := e_1 = \proj_K(p_0)$.
Then \cref{cor:diff_segments} and standard arguments yield
that $j := \frac12 \norm{\cdot}_H^2 + \delta_K$
is strongly twice epi-differentiable in $\bar x$ for $p_0$
and that $p_0 \in \partial j(\bar x)$.
In particular, 
it holds $Q_j^{\bar x, p_0}(z) \ge 0$ for all $z \in H$ by 
 \cref{lem:basic_properties}. For a fixed $L > 2$, on the other hand, we have
$-L\,t_n\,e_1 + e_2 + L \, e_n \weakly e_2$,
$\bar x + t_n \, (-L\,t_n\,e_1 + e_2 + L \, e_n) \in K$ for $n > 2$
and
\begin{equation*}
	-\frac2{t_n} \scalarprod{p_0 - \bar x}{-L\,t_n\,e_1 + e_2 + L \, e_n}_H
	=
	-L.
\end{equation*}
The above implies $Q_{\delta_K}^{\bar x, p_0 - \bar x}(e_2) = -\infty$. This shows that 
\eqref{eq:epi_umrechnen} indeed cannot be satisfied.

\subsection{Bang-Bang Optimal Control Problems}
\label{subsec:bang_bang}
In what follows, we consider  optimal control problems of the form
\begin{equation}
	\label{eq:bang-bang}
	\begin{aligned}
		\text{Minimize} \quad & F(u) := \int_\Omega L( \cdot, G(u)) \, \d x \\
		\text{such that} \quad & -1 \le u \le 1 \text{ a.e.\ in } \Omega.
	\end{aligned}
\end{equation}
Here,
$\Omega \subset \R^d$ is a  bounded domain
and $G : L^2(\Omega) \to L^2(\Omega)$ is the control-to-state map.
The function $F : L^2(\Omega) \to \R$ is called the reduced objective functional.
We are interested in stability properties of a bang-bang solution
$\bar u$ of \eqref{eq:bang-bang},
i.e., a local solution of \eqref{eq:bang-bang} which satisfies $\bar u \in \{-1,1\}$ a.e.\ on $\Omega$.

In order to give a more tangible example,
we will mainly focus on the optimal control of a semilinear partial differential equation (PDE) in the case $d \in \{1,2,3\}$.
We emphasize that 
our arguments also apply to a broader class of problems,
see \cref{rem:bang_bang_more_applications} below.

So let us consider a bounded domain $\Omega \subset \R^d$, $d \in \{1,2,3\}$,
with a Lipschitz boundary.
We assume that the state $G(u)$ associated with the control 
$u \in \Uad  := \{ v \in L^\infty(\Omega) \mid -1 \le v \le 1 \text{ a.e.\ in } \Omega\}$ is defined to be the weak solution $y \in H_0^1(\Omega)$ of  the PDE
\begin{equation}
	\label{eq:pde}
	-\Delta y + f(\cdot, y) = u \text{ in }\Omega,
	\qquad
	y = 0 \text{ on } \partial\Omega.
\end{equation}
We further suppose that the functions  $f$ and $L$ appearing in \eqref{eq:bang-bang} and \eqref{eq:pde} satisfy
the following
common assumptions,
cf.\ \cite{Casas2012:1,CasasWachsmuthWachsmuth2016:1,NguyenWachsmuth2017:1}.

\begin{assumption}
	\label{asm:on_f_L}
The functions
$L, f : \Omega \times \R \to \R$
are Carathéodory mappings that are
$C^2$ w.r.t.\ their second argument.
Moreover, the following conditions hold true.
	\begin{enumerate}
		\item
			We have $L(\cdot, 0) \in L^1(\Omega)$.
			For each $M > 0$, there is a constant $C_{L,M}$ and a function $\psi_M \in L^2(\Omega)$ such that
			\begin{gather*}
				\Bigabs{ \frac{\partial L}{\partial y}(x,y) } \le \psi_M(x), \qquad
				\Bigabs{ \frac{\partial^2 L}{\partial y^2}(x,y) } \le C_{L,M},
				\\
				\Bigabs{ \frac{\partial^2 L}{\partial y^2}(x,y_1) - \frac{\partial^2 L}{\partial y^2}(x,y_2) } \le C_{L,M} \, \abs{y_1 - y_2}
			\end{gather*}
			hold for a.a.\ $x \in \Omega$ and all $\abs{y}, \abs{y_2}, \abs{y_2} \le M$.
		\item
			We have $f(\cdot, 0) \in L^2(\Omega)$ and $\frac{\partial f}{\partial y}(x,y) \ge 0$ for a.a.\ $x \in \Omega$ and all $y \in \R$.
			For each $M > 0$, there is a constant $C_{f,M}$ such that
			\begin{align*}
				\Bigabs{ \frac{\partial f}{\partial y}(x,y) } +
				\Bigabs{ \frac{\partial^2 f}{\partial y^2}(x,y) } &\le C_{f,M},
				&
				\Bigabs{ \frac{\partial^2 f}{\partial y^2}(x,y_1) - \frac{\partial^2 f}{\partial y^2}(x,y_2) } &\le C_{f,M} \, \abs{y_1 - y_2}
			\end{align*}
			hold for a.a.\ $x \in \Omega$ and all $\abs{y}, \abs{y_2}, \abs{y_2} \le M$.
	\end{enumerate}
\end{assumption}
We remark that it is possible to weaken the Lipschitz assumption on the second derivatives of $f$ and $L$ in the above 
to a kind of uniform continuity, see, e.g., \cite[Assumptions~(A1), (A2)]{NguyenWachsmuth2017:1}.
We will not pursue this approach here. 

Note that the conditions in \cref{asm:on_f_L} imply
the following differentiability results for
the control-to-state map $G$ and the reduced objective $F$.

\begin{lemma}
\label{lem:diff_bang_bang}
The control-to-state map $G : L^2(\Omega) \to L^2(\Omega)$ and the reduced objective $F : L^2(\Omega) \to \R$ are well-defined and 
twice continuously Fréchet differentiable.
For all \mbox{$u \in L^2(\Omega)$},
the first derivative of $F$ satisfies 
$F'(u) \in C_0(\Omega)$
and the second derivative $F''(u)$ can be extended from $L^2(\Omega)$
to a continuous bilinear form on $\MM(\Omega)$.
Moreover, it holds $F''(u)\mu \in C_0(\Omega)$ for all $\mu \in \MM(\Omega)$
and
$\mu_k \weaklystar \mu$ in $\MM(\Omega)$
implies $F''(u) \, \mu_k \to F''(u) \, \mu$ in $C_0(\Omega)$.
For all $u \in \Uad$, the derivative $G'(u) : L^2(\Omega) \to L^2(\Omega)$
can be extended to a mapping $G'(u) : \MM(\Omega) \to L^2(\Omega)$,
which maps weakly-$\star$ convergent sequences to strongly convergent sequences,
and there exists an adjoint $G'(u)\adjoint : L^2(\Omega) \to C_0(\Omega)$ satisfying
\begin{equation}
	\label{eq:adjoint_linearized}
	\scalarprod{p}{G'(u) \, \mu}_{L^2(\Omega)}
	=
	\dual{G'(u)\adjoint p}{\mu}
	\qquad\forall p \in L^2(\Omega), \mu \in \MM(\Omega).
\end{equation}

The mappings $F$ and $G$ are compact in the sense that
for
all small $p_1 \in L^2(\Omega)$
and
all sequences $\{u_k\} \subset \Uad$
with $u_k \weaklystar u$ in $L^\infty(\Omega)$,
we have
$F(u_k) \to F(u)$ in $\R$ and
$G(u_k + p_1) \weakly G(u + p_1)$ in $L^2(\Omega)$.

There is a $\delta > 0$ such that for some $C > 0$, the estimates
\begin{align}
	\label{eq:diff_bang_bang_1}
	\norm{ G'(u_1 + p_1)v - G'(u_2)v }_{L^2(\Omega)}
	&
	\le C \, \norm{u_1 + p_1 - u_2}_{L^2(\Omega)} \, \norm{v}_{L^1(\Omega)}
	\\
	\label{eq:diff_bang_bang_2}
	\norm{ G'(u_1)v }_{L^2(\Omega)}
	&\le C \, \norm{v}_{L^1(\Omega)}
	\\
	\label{eq:diff_bang_bang_4}
	\bigabs{ F''(u_1 + p_1) [v_1, v_2] }
	&\le
	C \, \norm{v_1}_{L^1(\Omega)} \, \norm{v_2}_{L^1(\Omega)}
\end{align}
hold
for all $u_1, u_2 \in \Uad$, $v,v_1,v_2 \in L^2(\Omega)$ and $p \in L^2(\Omega)$ with $\norm{p}_{L^2(\Omega)} \le \delta$.

For every $\varepsilon > 0$, there is $\delta > 0$ such that
\begin{equation}
	\label{eq:diff_bang_bang_3}
	\bigabs{ F''(u_1 + p_1) [v_1, v_2] - F''(u_2) [v_1, v_2] }
	\le
	\varepsilon \, \norm{v_1}_{L^1(\Omega)} \, \norm{v_2}_{L^1(\Omega)}
\end{equation}
for all $u_1, u_2 \in \Uad$, $v_1,v_2 \in L^2(\Omega)$, $p \in L^2(\Omega)$
with
$\norm{u_1 - u_2}_{L^2(\Omega)} \le \delta$ and
$\norm{p}_{L^2(\Omega)} \le \delta$.
\end{lemma}
\begin{proof}
The well-definedness and the differentiability of $G$ and $F$ follow  from standard arguments,
see, e.g., \cite[p.~2357]{Casas2012:1}.
In particular,
the derivatives of the reduced objective are given by
\begin{align}
	F'(u)v &= \int_\Omega \varphi_u v \, \d x \label{E2.5},
	\\
	F''(u)(v_1,v_2) &= \int_\Omega\Big[\frac{\partial^2L}{\partial
		y^2}(\cdot,G(u)) - \frac{\partial^2f}{\partial
	y^2}(\cdot,G(u))\,\varphi_u\Big]\,G'(u)v_1 \, G'(u)v_2 \, \d x,
	\label{E2.6}
\end{align}
where $\varphi_u \in H_0^1(\Omega) \cap C_0(\Omega)$ is the adjoint state associated with $u$,
	i.e., the weak solution of
	\begin{equation}
		\label{eq:adjoint_PDE}
		-\Delta\varphi_u +
		\frac{\partial f}{\partial y}(\cdot,y_u)\,\varphi_u  =
		\displaystyle\frac{\partial L}{\partial y}(\cdot,y_u)
		\text{ in } \Omega,
		\qquad
		\varphi_u = 0 \text{ on }\partial\Omega.
	\end{equation}
	This implies $F'(u) = \varphi_u \in C_0(\Omega)$.

	Next, we discuss the properties of $G'(u)$.
	For $v \in L^2(\Omega)$, the derivative
	$z_{u,v} := G'(u) \, v$ is the unique weak solution of the linearized equation
	\begin{equation}
		\label{eq:linearized_equation}
		-\Delta z_{u,v} +
		\frac{\partial f}{\partial y}(\cdot,y_u)\, z_{u,v}  =
		v
		\text{ in } \Omega,
		\qquad
		z_{u,v} = 0 \text{ on }\partial\Omega.
	\end{equation}
	It follows from classical arguments that this PDE with right-hand side $\mu \in \MM(\Omega)$ has a unique solution
	$z_{u,\mu} \in L^2(\Omega)$
	for every $\mu \in \MM(\Omega)$
	and that $\mu_n \weaklystar \mu$ in $\MM(\Omega)$
	implies $z_{u,\mu_n} \to z_{u,\mu}$ in $L^2(\Omega)$,
	see \cite[Section~2.5]{CasasWachsmuthWachsmuth2016:1}.
	Finally, $G'(u)\adjoint p \in C_0(\Omega)$ can be defined as the weak solution of \eqref{eq:linearized_equation}
	with right-hand side $p \in L^2(\Omega)$
	and the desired formula \eqref{eq:adjoint_linearized} follows easily.

	Now, from \eqref{eq:adjoint_linearized} and \eqref{E2.6} it can be seen that $F''(u)$ satisfies
	\begin{align*}
		F''(u)(v_1,v_2) &=
		\int_\Omega
		\zeta_{u,v_1} \, v_2 \, \d x,
	\end{align*}
	where
	\begin{equation*}
		\zeta_{u,v_1}
		=
		G'(u)\adjoint\biggh(){
			\Big[\frac{\partial^2L}{\partial y^2}(\cdot,G(u)) - \frac{\partial^2f}{\partial y^2}(\cdot,G(u))\,\varphi_u\Big]
			\,G'(u)v_1
		}.
	\end{equation*}
	Hence,
	$F''(u) \, \mu = \zeta_{u,\mu} \in C_0(\Omega)$ holds for all $\mu \in \MM(\Omega)$
	and
	$\mu_k \weaklystar \mu$
	implies $F''(u) \, \mu_k = \zeta_{u,\mu_k} \to \zeta_{u,\mu} = F''(u)\,\mu$ in $C_0(\Omega)$.

	The asserted compactness of $F$ and $G$ can be shown as in \cite[Theorem~4.1]{NguyenWachsmuth2017:1}.

	The estimates \eqref{eq:diff_bang_bang_1}--\eqref{eq:diff_bang_bang_2} are proven in \cite[Lemma~4.2]{NguyenWachsmuth2017:1}.
	Estimate \eqref{eq:diff_bang_bang_4} follows from \eqref{E2.6}, \eqref{eq:diff_bang_bang_2}
	and the fact that $G(u_1 + p_1)$ and $\varphi_{u_1+p_1}$ can be uniformly bounded in $L^\infty(\Omega)$,
	see \cite[Lemma~4.1]{NguyenWachsmuth2017:1}.
	In \cite[Lemma~4.3]{NguyenWachsmuth2017:1},
	\eqref{eq:diff_bang_bang_3} is shown for $v_1 = v_2$.
	The general case follows by polarization.
	Indeed, with $B := F''(u_1 + p_1) - F''(u_2)$, we obtain
	\begin{align*}
		\abs{B(v_1, v_2)}
		&=
		\frac14\bigabs{B(v_1 + v_2, v_1 + v_2) - B(v_1 - v_2, v_1 - v_2)}
		\\
		&\le \frac\varepsilon4 \bigh(){ \norm{v_1 + v_2}_{L^1(\Omega)}^2 + \norm{v_1 - v_2}_{L^1(\Omega)}^2}
		= \frac\varepsilon2 \bigh(){ \norm{v_1}_{L^1(\Omega)}^2 + \norm{v_2}_{L^1(\Omega)}^2 }.
	\end{align*}
	With scaling, we can resort to the case $\norm{v_1}_{L^1(\Omega)} = \norm{v_2}_{L^1(\Omega)} = 1$
	and this yields \eqref{eq:diff_bang_bang_3}.
\end{proof}
\begin{remark}
	\label{rem:bang_bang_more_applications}
	In the sequel, we will only work with the results of \cref{lem:diff_bang_bang}.
	This means that the following theory is applicable to all problems
	for which the same estimates are available.
	In particular, in all of the following results, \cref{asm:on_f_L}
	can be substituted by the assertions of \cref{lem:diff_bang_bang}.
\end{remark}

Next, we are going to apply the second-order theory of \cite{ChristofWachsmuth2017:1}
to the optimal control problem \eqref{eq:bang-bang}.
Therefore, let $\bar u \in \Uad$ be a stationary point of \eqref{eq:bang-bang},
i.e.,
$\bar \varphi := F'(\bar u) \in C_0(\Omega)$
satisfies
\begin{equation*}
	\scalarprod{\bar\varphi}{u - \bar u}_{L^2(\Omega)} \ge 0
	\qquad\forall u \in \Uad.
\end{equation*}
Note that, in the setting of the semilinear PDE \eqref{eq:pde}, $\bar\varphi$ is the solution of the adjoint equation
\eqref{eq:adjoint_PDE} with $y_u$ replaced by $\bar y = G(\bar u)$, see \eqref{E2.5}.

Let us introduce some notation in order to comply with the setting of
\cite{ChristofWachsmuth2017:1}
and of \cref{sec:abstract_analysis}.
We define the separable space
\begin{equation*}
	Y := C_0(\Omega) =  \mathrm{cl}_{\|.\|_\infty} \left ( C_c(\Omega) \right )
\end{equation*}
endowed with the usual supremum norm.
Its dual space can be identified with
\begin{equation*}
	X := \MM(\Omega)
\end{equation*}
which is the space of signed finite Radon measures on the domain $\Omega$
endowed with the norm $\|\mu\|_{\MM(\Omega)}:=|\mu|(\Omega)$, cf.\ \cite[Theorem 1.54]{Ambrosio}.
The space $L^1(\Omega)$ is identified with a closed subspace of $\MM(\Omega)$ via the isometric embedding $h \mapsto h \LL^d$, where $\LL^d$ is Lebesgue's measure.
In the same way,
$\Uad$ is considered as a subset of $\MM(\Omega)$.

\begin{assumption}[Assumptions for the Calculation of the Second Subderivative]
\label{assumption:bangbangsolution} 
We require
$\bar \varphi \in C_0(\Omega) \cap C^1(\Omega)$
and define
$\ZZ := \{z \in \Omega : \bar \varphi(z) = 0\}$.
We assume $\ZZ \subset \{z \in \Omega : \abs{\nabla \bar \varphi (z)} \neq 0 \}$.
Here and in the sequel, $\abs{\nabla \bar \varphi(z)}$
denotes the Euclidean norm of $\nabla \bar\varphi(z) \in \R^d$.
We further require
\begin{equation}
	\label{eq:measure_of_set}
	c
	:=
	\liminf_{s \searrow 0} \left ( \frac{s}{\LL^d(\{ |\bar\varphi| \leq s\})} \right )
	>
	0.
\end{equation}

\end{assumption}

Note that \eqref{eq:measure_of_set} is in particular satisfied (with $c \ge C^{-1}$) if
\begin{equation}
	\label{eq:measure_of_set_reloaded}
	\LL^d(\{\abs{\bar\varphi} \le s\}) \le C \, s
	\qquad\forall s > 0
\end{equation}
holds for some $C > 0$.
Such an assumption was previously used in, e.g., \cite{WachsmuthWachsmuth2011,DeckelnickHinze2012,CasasWachsmuthWachsmuth2016:1}.

We are now in the position to
study the second-order epi-differentiability of
the indicator function of $\Uad$.
In what follows, we denote by
$\HH^{d-1}$ the $(d-1)$-dimensional Hausdorff measure, which is scaled as in \cite[Definition 2.1]{EvansGariepy}.
\begin{theorem}
	\label{thm:curvature}
	Under \cref{assumption:bangbangsolution},
	the indicator function
	$j := \delta_{\Uad}$
	of $\Uad$
	is strictly twice epi-differentiable
	in $\bar u$ for $-\bar\varphi$
	with
	\begin{equation*}
		\KK\jxp = 
		\varh\{\}{
			g \mathcal{H}^{d-1}|_\ZZ \mid g \in L^1\left ( \ZZ,  \mathcal{H}^{d-1} \right) \cap   L^2 \left ( \ZZ, \abs{\nabla \bar \varphi} \mathcal{H}^{d-1} \right)
		}
	\end{equation*}
	and
	for every element $h = g \mathcal{H}^{d-1}|_\ZZ$ of the above set
	we have
	\begin{equation*}
		Q\jxp(h)
		=
		\frac{1}{2}  \int_{\ZZ}\  g^2 \abs{\nabla \bar \varphi } \mathrm{d}\mathcal{H}^{d-1}.
	\end{equation*}
\end{theorem}
\begin{proof}
	In view of
	\begin{align*}
		Q\jxp(z)
		&=
		\inf
		\biggh\{\}{
			\liminf_{n \to \infty} \frac{j(\bar u + t_n \, z_n) - j(\bar u) + t_n \dual{\bar\varphi}{z_n}}{t_n^2/2}
			\mid
			t_n \searrow 0,
			z_n \weaklystar z
		}
		\\
		&=
		\inf
		\biggh\{\}{
			\liminf_{n \to \infty} \frac{\dual{\bar\varphi}{z_n}}{t_n/2}
			\mid
			t_n \searrow 0,
			z_n \weaklystar z,
			\bar u + t_n \, z_n \in \Uad
		}
		\\
		&=
		\inf
		\biggh\{\}{
			\liminf_{n \to \infty} \dual{\bar\varphi}{2 z / t_n + r_n}
			\mid
			t_n \searrow 0,
			r_n \weaklystar 0,
			\bar u + t_n \, z + \frac12 \, t_n^2 \, r_n \in \Uad
		},
	\end{align*}
	we find that $Q\jxp(z) = +\infty$ for $z \not\in \TT^\star_{\Uad}(\bar u) \cap \bar\varphi\anni$,
	where
	\begin{equation*}
		\TT_{\Uad}^{\star}(\bar u)
		:=
		\varh\{\}{ h \in \MM(\Omega) \mid \exists t_k \searrow 0, \exists u_k \in \Uad \text{ such that } \frac{u_k - \bar u}{t_k} \weaklystar h }.
	\end{equation*}
	Moreover, for all $z \in \TT^\star_{\Uad}(\bar u) \cap \bar\varphi\anni$
	the value
	$Q\jxp(z)$ coincides with the value of the curvature functional introduced in \cite{ChristofWachsmuth2017:1},
	see in particular the comment after \cite[Definition~3.1]{ChristofWachsmuth2017:1}.
	Now,
	\cite[Theorem~6.11]{ChristofWachsmuth2017:1}
	shows
	that the weak-$\star$ second subderivative of $j$
	together with its domain $\KK\jxp$
	is given as in the assertion of the theorem.

	The strict second-order epi-differentiability of $j$ in $(\bar u, -\bar\varphi)$
	follows
	from
	\cite[Lemma~6.10]{ChristofWachsmuth2017:1}
	combined with
	\cref{lem:obtain_twice_epi}, cf.\ the proof of \cite[Theorem~6.11]{ChristofWachsmuth2017:1}.
\end{proof}
For later use, we remark that the reduced critical cone
$\KK\jxp$ is a linear subspace of $\MM(\Omega)$.
Moreover, $Q\jxp$ is a quadratic functional on this subspace.
In particular, we can define the bilinear form associated with $Q\jxp$
via
\begin{equation*}
	\hat Q\jxp[h_1, h_2]
	=
	\frac{1}{2}  \int_{\ZZ}\  g_1 \, g_2 \abs{\nabla \bar \varphi } \mathrm{d}\mathcal{H}^{d-1}
	\in \R
\end{equation*}
for all
$h_1 = g_1 \mathcal{H}^{d-1}|_\ZZ \in \KK\jxp$
and
$h_2 = g_2 \mathcal{H}^{d-1}|_\ZZ \in \KK\jxp$.

Finally, in order to apply the second-order theory,
we have to verify \cite[Assumption~4.1]{ChristofWachsmuth2017:1}.
This amounts to the verification of
\begin{equation}
	\label{eq:hadamard_taylor_expansion}
	\lim_{k \to \infty}
	\frac{F(\bar u + t_k \, h_k) - F(\bar u) - t_k \, F'(\bar u) \, h_k - \frac12 \, t_k^2 \, F''(\bar u) \, h_k^2}{t_k^2}
	=
	0
\end{equation}
for all $\{h_k\} \subset L^\infty(\Omega)$,
$\{t_k\} \subset \mathbb{R}^+$ satisfying $t_k \searrow 0$,
$h_k \weaklystar h \in \MM(\Omega)$ and $\bar u + t_k h_k \in \Uad$.
In order to verify \eqref{eq:hadamard_taylor_expansion},
we use the Taylor expansion
\begin{align*}
	0
	&=
	F(\bar u + t_k \, h_k) - F(\bar u) - t_k \, F'(\bar u) \, h_k - \frac12 \, t_k^2 \, F''(u_k) \, h_k^2
	\\
	&=
	F(\bar u + t_k \, h_k) - F(\bar u) - t_k \, F'(\bar u) \, h_k
	- \frac12 \, t_k^2 \, F''(\bar u) \, h_k^2
	+ \frac12 \, t_k^2 \, \bigh[]{ F''(\bar u) \, h_k^2 -  F''(u_k) \, h_k^2}
\end{align*}
with $u_k = \bar u + \tau_k \, t_k \, h_k \in \Uad$ for some $\tau_k \in [0,1]$.
We recall that  $L^\infty(\Omega) \ni h_k \weaklystar h$ in $\MM(\Omega)$ implies that $h_k$ is bounded in $L^1(\Omega)$.
Using that $u_k - \bar u = \tau_k \, t_k \, h_k$
is a null sequence in $L^1(\Omega)$ and bounded in $L^\infty(\Omega)$,
we have $\norm{u_k - \bar u}_{L^2(\Omega)} \to 0$.
For $\varepsilon > 0$ by using \eqref{eq:diff_bang_bang_3}
we
obtain
\begin{equation*}
	\frac{\bigabs{
			F(\bar u + t_k \, h_k) - F(\bar u) - t_k \, F'(\bar u) \, h_k
			- \frac12 \, t_k^2 \, F''(\bar u) \, h_k^2
	}}{t_k^2}
	\le
	\varepsilon \, \norm{h_k}_{L^1(\Omega)}^2
\end{equation*}
for $k$ large enough.
Due to the boundedness of $h_k$ in $L^1(\Omega)$,
this implies \eqref{eq:hadamard_taylor_expansion}.

Now we can provide a second-order condition.
\begin{theorem}
	\label{thm:ssc}
	Under \cref{asm:on_f_L,assumption:bangbangsolution},
	the condition
	\begin{equation}
		\label{eq:explicitbangbang}
		F''(\bar u) h^2  + Q\jxp(h)
		>
		0
		\quad
		\forall  h \in  \KK\jxp \setminus \{0\}
	\end{equation}
	is equivalent to the quadratic growth condition
	\begin{equation}
	\label{eq:bangbanggrowth}
	F(u) \ge F(\bar u) + \frac{c}{2} \, \norm{ u - \bar u }_{L^1(\Omega)}^2\quad\forall u \in \Uad , \norm{u - \bar u}_{L^1(\Omega)} \le \varepsilon
	\end{equation}
	with constants $c>0$ and $\varepsilon > 0$.
	Further, \eqref{eq:bangbanggrowth}
	implies
	\begin{equation}
		\label{eq:explicitbangbang_coercive}
		F''(\bar u) h^2  + Q\jxp(h)
		\ge
		c \, \norm{h}^2_{\MM(\Omega)}
		\quad
		\forall  h \in  \KK\jxp
		.
	\end{equation}
\end{theorem}
\begin{proof}
	The result follows from
	\cite[Theorem~6.12]{ChristofWachsmuth2017:1}
	since all assumptions are verified in our situation,
	see also
	\cite[Theorem~4.3]{ChristofWachsmuth2017:1}.
\end{proof}
We mention that the appearance of $Q\jxp(h)$
in \eqref{eq:explicitbangbang} and \eqref{eq:explicitbangbang_coercive}
accounts for the curvature of the set $\Uad$
w.r.t.\ the weak-$\star$ topology of $\MM(\Omega)$.
It is surprising that $\Uad$ possesses curvature in this situation,
since it is well known that $\Uad$ is a polyhedric set in the
function spaces $L^p(\Omega)$, i.e., it does not possess any curvature in these stronger topologies.

In case that the growth condition \eqref{eq:bangbanggrowth}
is satisfied,
we expect that the solution $\bar u$ is stable
w.r.t.\ small perturbations of the objective $F$
and we are interested in its sensitivity properties.
To this end, we consider a class of perturbations as in
\cite{NguyenWachsmuth2017:1}.
We fix the perturbation space
\begin{equation*}
	P = L^2(\Omega) \times L^2(\Omega)
\end{equation*}
and for a perturbation $p = (p_1, p_2) \in P$,
we define the perturbed objective
\begin{equation*}
	J(p,u)
	=
	F(u + p_1) + \scalarprod{p_2}{G(u + p_1)}_{L^2(\Omega)}
	.
\end{equation*}
In what follows, we first
address the solvability of the perturbed problem and the stability of $\bar u$
w.r.t.\ perturbations.
\begin{theorem}
	\label{thm:lipschitz_solvability_bang_bang}
	Suppose that \cref{asm:on_f_L} and \eqref{eq:bangbanggrowth}
	hold with some $c, \varepsilon > 0$.
	Then, there is $\delta > 0$ such that for all $\norm{p}_P \le \delta$ the perturbed problem
	\begin{equation}
		\label{eq:perturbed}
		\text{Minimize } J(p, u) \text{ w.r.t.\ } u \in \Uad
	\end{equation}
	has a local solution $\bar u_p$ satisfying $\norm{\bar u_p  - \bar u}_{L^1(\Omega)} \le \varepsilon$
	and $J(p,\bar u_p) \le J(p, \bar u)$. Moreover, there exists a constant $L>0$ such that 
	for all $\bar u_p$ with the latter two properties, it holds
	\begin{equation*}
		\norm{ \bar u_p - \bar u}_{L^1(\Omega)}
		\le
		L \, \norm{p}_P.
	\end{equation*}
\end{theorem}
\begin{proof}
	Let us take some $\delta \le 1$
	and $p \in P$ with $\norm{p}_P \le \delta$.
	Using a Taylor expansion, we find
	\begin{equation*}
		J(p,u)
		=
		F(u + p_1)
		+
		\scalarprod{p_2}{G(u + p_1)}
		=
		F(u) + F'(u + \tau \, p_1) \, p_1
		+
		\scalarprod{p_2}{G(u + p_1)}
	\end{equation*}
	for all $u \in \Uad$ with a $\tau \in [0,1]$.
	We use an analogous expansion for $J(p, \bar u)$.
	Together with \eqref{eq:bangbanggrowth}, this yields
	\begin{align*}
		J(p, u) - J(p, \bar u)
		&=
		F(u) + F'(u + \tau \, p_1) \, p_1
		-
		F(\bar u) - F'(\bar u + \bar\tau \, p_1) \, p_1
		\\
		&\qquad
		+
		\scalarprod{p_2}{G(u + p_1) - G(\bar u + p_1)}
		\\
		&\ge
		\frac{c}{2} \, \norm{u - \bar u}_{L^1(\Omega)}^2
		-
		\bigabs{
		F'(u + \tau \, p_1) \, p_1
		- F'(\bar u + \bar\tau \, p_1) \, p_1
		}
		\\
		&\qquad
		-
		\norm{p_2}_{L^2(\Omega)} \, \norm{G(u + p_1) - G(\bar u + p_1)}_{L^2(\Omega)}
	\end{align*}
	for all $u \in \Uad$ with $\norm{u - \bar u}_{L^1(\Omega)} \le \varepsilon$.
	Now, we further suppose that $\delta$ is small enough such that \eqref{eq:diff_bang_bang_2} applies.
	This inequality yields
	\begin{equation*}
		\norm{G(u + p_1) - G(\bar u + p_1)}_{L^2(\Omega)}
		\le
		\norm{G'(\lambda \, u + (1-\lambda) \, \bar u + p_1)(u - \bar u)}_{L^2(\Omega)}
		\le
		C \, \norm{u - \bar u}_{L^1(\Omega)}
	\end{equation*}
	with some $\lambda \in [0,1]$.
	Using \eqref{eq:diff_bang_bang_4}, the term involving $F'$ gives
	\begin{align*}
		\bigabs{
		F'(u + \tau \, p_1) \, p_1
		- F'(\bar u + \bar\tau \, p_1) \, p_1
		}
		&=
		\bigabs{
			F''( \hat u) [ u - \bar u + (\tau - \bar\tau) \, p_1, p_1]
		}
		\\
		& \le
		C \, \norm{u - \bar u + (\tau - \bar\tau) \, p_1}_{L^1(\Omega)} \, \norm{p_1}_{L^2(\Omega)}.
	\end{align*}
	Together with the above estimates and Young's inequality
	we arrive at
	\begin{equation*}
		J(p,u) - J(p,\bar u)
		\ge
		\frac{c}{4}\,\norm{u - \bar u}_{L^1(\Omega)}^2
		-
		C \, \norm{p}_P^2.
	\end{equation*}
	Now, we further suppose that
	$\delta \le \varepsilon \sqrt{c / (4C)}$.
	For all $u \in \Uad$
	with $\norm{u - \bar u}_{L^1(\Omega)} \in (\sqrt{(4C)/c} \, \norm{p}, \varepsilon]$
	the above calculation shows
	$J(p,u) > J(p,\bar u)$.
	Hence, we can utilize the compactness of $F$ and $G$, see \cref{lem:diff_bang_bang},
	and a standard argument shows that the perturbed problem
	has at least one solution $\bar u_p$ in the $\varepsilon$-ball centered at $\bar u$.
	Moreover, all these solutions satisfy the desired inequalities with
	$L = \sqrt{(4C)/c}$.
\end{proof}
Similarly, we can provide a stability result for stationary points.
\begin{theorem}
	\label{thm:lipschitz_solvability_KKT_bang_bang}
	Suppose that {\renewcommand\crefpairconjunction{, }\cref{asm:on_f_L,assumption:bangbangsolution}} and \eqref{eq:bangbanggrowth}
	hold with some $c, \varepsilon > 0$.
	Then, there exist $\delta,\eta,L > 0$
	such that for all $\norm{p}_P \le \delta$ and all stationary points $\bar u_p$ for the perturbed problem
	\begin{equation*}
		\text{Minimize } J(p, u) \text{ w.r.t.\ } u \in \Uad
	\end{equation*}
	with 
	$\norm{\bar u_p - \bar u}_{L^1(\Omega)} \le \eta$
	we have
	\begin{equation*}
		\norm{ \bar u_p - \bar u}_{L^1(\Omega)}
		\le
		L \, \norm{p}_P.
	\end{equation*}
\end{theorem}
\begin{proof}
	We proceed by contradiction.
	This yields sequences
	$p_k \to 0$ in $P$,
	$\bar u_k \to \bar u$ in $L^1(\Omega)$
	such that
	$\bar u_k$ is a stationary point for the perturbed problem with $p = p_k$
	and
	$\norm{\bar u_k - \bar u}_{L^1(\Omega)} \ge k \, \norm{p_k}_P$.
	We set $t_k := \norm{\bar u_k - \bar u}_{L^1(\Omega)}$
	and w.l.o.g.\ $v_k := (\bar u_k - \bar u) / t_k \weaklystar \mu$ in $\MM(\Omega)$.
	For arbitrary $\hat\mu \in \KK\jxp$, let $\hat\mu_k$ be a recovery sequence,
	i.e., $\bar u + t_k \, \hat\mu_k \in \Uad$, $\hat\mu_k \weaklystar \hat\mu$ in $\MM(\Omega)$
	and $Q\jxp(\hat\mu) = \lim_{k\to\infty} \frac{\dual{\bar\varphi}{\hat\mu_k}}{t_k/2}$.
	Since $\bar u_k$ is stationary for $p_k$,
	we have
	\begin{equation*}
		\dual{J_u(p_k, \bar u_k)}{\bar u + t_k \, \hat\mu_k - \bar u_k} \ge 0.
	\end{equation*}
	Using \eqref{eq:diff_bang_bang_1} and \eqref{eq:diff_bang_bang_3},
	a tedious computation shows
	$\norm{J_u(p_k, \bar u_k) - J_u(0, \bar u_k)}_{C_0(\Omega)} \le C \, \norm{p_k}_P$
	for some $C > 0$ and all $k$ large enough.
	Note that the same estimate was derived in \cite[proof of Theorem~4.5]{NguyenWachsmuth2017:1} in the setting of the semilinear PDE.
	Using $J_u(0,\bar u_k) = F'(\bar u_k)$
	together with a Taylor expansion of $F'$
	and \eqref{eq:diff_bang_bang_3},
	this leads to
	\begin{equation*}
		0
		\le
		\dual{F'(\bar u)}{\bar u + t_k \, \hat \mu_k - \bar u_k}
		+
		F''(\bar u)[\bar u + t_k \, \hat \mu_k - \bar u_k, \bar u_k - \bar u]
		+
		\varepsilon_k \, t_k^2
		+
		\frac Ck \, t_k^2,
	\end{equation*}
	where $\varepsilon_k \to 0$.
	Dividing by $t_k^2$ yields
	\begin{equation}
		\label{eq:in_the_proof}
		\frac12 \, \frac{\dual{F'(\bar u)}{\frac{\bar u_k - \bar u}{t_k}}}{t_k/2}
		-
		\frac12 \, \frac{\dual{F'(\bar u)}{\hat\mu_k}}{t_k/2}
		+
		F''(\bar u) \, \biggh[]{\frac{\bar u_k - \bar u}{t_k} - \hat\mu_k, \frac{\bar u_k - \bar u}{t_k}}
		\le
		\varepsilon_k
		+
		\frac Ck.
	\end{equation}
	By passing to the limit $k \to \infty$, we obtain
	\begin{equation*}
		\frac12 \, Q\jxp(\mu) - \frac12 \, Q\jxp(\hat\mu)+ F''(\bar u) \, [\mu-\hat\mu,\mu] \le 0
		\qquad\forall \hat\mu \in \KK\jxp.
	\end{equation*}
	Since $\KK\jxp$ is a subspace,
	we can choose $\hat\mu = s \, \mu$ for $s \in (0,1) \cup (1,2)$.
	Dividing the above inequality by $s-1$ and passing to the limits $s \nearrow 1$ and $s \searrow 1$,
	this shows
	\begin{equation*}
		Q\jxp(\mu) + F''(\bar u) \, \mu^2 \le 0.
	\end{equation*}
	Now, \eqref{eq:explicitbangbang} implies $\mu = 0$.
	From \cite[Lemma~6.3]{ChristofWachsmuth2017:1}, we know that \eqref{eq:measure_of_set}
	implies
	\begin{equation*}
		\Bigdual{F'(\bar u)}{\frac{\bar u_k - \bar u}{t_k}}
		=
		\frac{\dual{F'(\bar u)}{\bar u_k - \bar u}}{t_k}
		\ge
		c \, \frac{\norm{\bar u_k - \bar u}^2}{t_k}
		=
		c \, t_k
	\end{equation*}
	for some $c > 0$.
	Using this information, we reconsider
	\eqref{eq:in_the_proof} with $\hat\mu_k \equiv 0$
	and
	together with $F''(\bar u)[(\bar u_k - \bar u)/t_k]^2 \to F''(\bar u) 0^2 = 0$
	this yields a contradiction.
\end{proof}
Note that, if we set $p = 0$, then the above theorem yields that $\bar u$ is the unique stationary point
for the unperturbed problem in $\Uad \cap B_\eta^{L^1(\Omega)}(\bar u)$.

We compare \cref{thm:lipschitz_solvability_bang_bang}
with \cite[Theorem~4.5]{NguyenWachsmuth2017:1}.
Therein, the authors used a relaxed version of the measure assumption \eqref{eq:measure_of_set_reloaded},
in which the right-hand side $C\,s$ is replaced by $C\,s^{\text{\ae}}$
for some $\text{\ae} > 0$.
However, their CQ \cite[Eq.~(3.22)]{NguyenWachsmuth2017:1}
(in conjunction with $\text{\ae} = 1$)
implies our growth condition,
see \cite[Theorem~3.1]{NguyenWachsmuth2017:1}.

The next theorem is the main theorem of this section.
\begin{theorem}
	\label{thm:sensitivity_bang_bang}
	Let \cref{asm:on_f_L,assumption:bangbangsolution} be satisfied.
	Suppose that \eqref{eq:bangbanggrowth} holds for some $c, \varepsilon > 0$.
	Let $q_t \to q$ in $P$ be given.
	For $t > 0$ small enough,
	denote by $\bar u_t$ a stationary point of \eqref{eq:perturbed} with perturbation $p = t \, q_t$
	satisfying $\bar u_t \to \bar u$ in $L^1(\Omega)$.
	Then, there is a measure $\mu \in \KK\jxp$ such that
	$(\bar u_t - \bar u)/t \weaklystar \mu$
	in $\MM(\Omega)$ as $t \searrow 0$.
	Moreover, $\mu$ is given by the unique solution in $\KK\jxp$ of
	\begin{equation}
		\label{eq:bang_bang_sensitivity_equation}
		\dual{J_{up}(0, \bar u) p + J_{uu}(0,\bar u) \, \mu}{\zeta}
		+
		\hat Q\jxp[\mu, \zeta]
		=0
		\qquad\forall \zeta \in \KK\jxp,
	\end{equation}
	and the mapping $p \mapsto \mu$ is linear.
\end{theorem}
\begin{proof}
	In a first step,
	we use the first-order necessary conditions
	to identify the optimizers $\bar u$ and $\bar u_t$ with
	a solution of a variational inequality.
	To this end,
	we use the definitions
	\begin{align*}
		A(p, u) := J_u(p, u),
		\qquad
		j  := \delta_{\Uad}.
	\end{align*}
	In particular, for $h \in \MM(\Omega)$ we have
	\begin{equation*}
		A(p, u) \, h
		=
		J_u(p, u) \, h
		=
		F'(u + p_1) \, h
		+
		\scalarprod{p_2}{G'(u + p_1) \, h}_{L^2(\Omega)}
		.
	\end{equation*}
	Now, the VI \eqref{eq:VI} is equivalent to
	\begin{equation*}
		\bar u_p \in \Uad,
		\qquad
		\dual{J_u(p,\bar u_p)}{v - \bar u_p} \ge 0
		\quad
		\forall v \in \Uad.
	\end{equation*}
	That is, every stationary point of the perturbed problem \eqref{eq:perturbed}
	is a solution of this VI.

	Let us check that the assumptions of \cref{thm:sufficient_abstract} are satisfied.
	The standing \cref{asm:data} is fulfilled by the above choices of
	$X = \MM(\Omega)$, $Y = C_0(\Omega)$, $P = L^2(\Omega)^2$,
	and $A : P \times X \to Y$.
	
	Part~\ref{assumption:sensitivity:i} of \cref{assumption:sensitivity} 
	is exactly our requirement $q_t \to q$ in $P$.
	Since $\bar u_t$ is a local solution of \eqref{eq:perturbed},
	it is a solution of the VI \eqref{eq:VI} with parameter $p = t \, q_t$
	and the Lipschitz estimate \eqref{eq:LipschitzEstimate}
	in
	\cref{assumption:sensitivity}~\ref{assumption:sensitivity:ii}
	follows from \cref{thm:lipschitz_solvability_KKT_bang_bang}.

	Finally, it remains to check
	the differentiability assumption
	\cref{assumption:sensitivity}~\ref{assumption:sensitivity:iii}.
	To this end,
	we define
	\begin{align*}
		a_0 &:= -J_u(0, \bar u), &
		\dual{A_p p}{h} &:= F''(\bar u)[h, p_1] + \scalarprod{p_2}{G'(\bar u) \, h}_{L^2(\Omega)}, \\
		&&
		\dual{A_x v}{h} &:= F''(\bar u)[h, v]
	\end{align*}
	for all $v, h \in \MM(\Omega)$ and $p \in P$.
	Note that these mappings satisfy the mapping properties of
	\cref{assumption:sensitivity}~\ref{assumption:sensitivity:iii},
	see \cref{lem:diff_bang_bang}.
	Moreover, we define the difference quotient
	\begin{equation*}
		v_t
		:=
		\frac{\bar u - \bar u_t}{t}
		\in
		L^\infty(\Omega)
	\end{equation*}
	which is bounded in $L^1(\Omega)$.
	Now, the residual $r(t) \in Y$ from \eqref{eq:taylor_A}
	is given by
	\begin{align*}
		\dual{r(t)}{h} &:=
		A(t \, q_t, \bar u + t \, v_t) \, h
		-
		A(0, \bar u) \, h
		-
		t \, \dual{A_p q_t}{h}
		-
		t \, \dual{A_x v_t}{h}
		\\&
		=
		F'(\bar u + t \, (q_{t,1} + v_t)) \, h - F'(\bar u) \, h
		- t \, F''(\bar u) [h, q_{t,1} + v_t]
		\\&\qquad
		+ \bigscalarprod{t \, q_{t,2}}{G'(\bar u + t \, (q_{t,1} + v_t)) \, h}_{L^2(\Omega)}
		- t \, \bigscalarprod{q_{t,2}}{G'(\bar u) \, h}_{L^2(\Omega)}
		,
	\end{align*}
	where $h \in X$.
	In order to obtain an estimate of
	$\norm{r(t)}_Y$, it is sufficient to test $r(t) \in Y = C_0(\Omega)$
	with functions $h \in L^1(\Omega)$.
	Using a Taylor expansion of $F$, we find
	\begin{align*}
		t^{-1} \, \dual{r(t)}{h}
		&=
		F''\bigh(){ \bar u + \tau_t \, t \, (q_{t,1} + v_t)}[h, q_{t,1} + v_t]
		-
		F''(\bar u)[h, q_{t,1} + v_t]
		\\&\qquad
		+
		\bigscalarprod{q_{t,2}}{G'(\bar u + t \, (q_{t,1} + v_t)) \, h - G'(\bar u) \, h}_{L^2(\Omega)}
		\\
		&=
		\Bigh[]{F''\bigh(){ \bar u + \tau_t \, t \, (q_{t,1} + v_t)} - F''(\bar u)}[h, q_{t,1} + v_t]
		\\&\qquad
		+
		\bigscalarprod{q_{t,2}}{G'(\bar u + t \, (q_{t,1} + v_t)) \, h - G'(\bar u) \, h}_{L^2(\Omega)},
	\end{align*}
	where $\tau_t \in [0,1]$.
	Now, we are going to use the estimates \eqref{eq:diff_bang_bang_1} and \eqref{eq:diff_bang_bang_3}.
	Therefore,
	we utilize
	$\bar u + \tau_t \, t \, (q_{t,1} + v_t) \to \bar u$ in $L^2(\Omega)$.
	In particular, for every $\varepsilon > 0$,
	there is $\hat t > 0$,
	such that we can apply \eqref{eq:diff_bang_bang_3} with
	$u_1 = \bar u + \tau_t \, t \, v_t$,
	$u_2 = \bar u$,
	and $p_1 = \tau_t \, t \, q_{t,1}$.
	This leads to the estimate
	\begin{equation*}
		t^{-1} \, \bigabs{ \dual{r(t)}{h} }
		\le
		\varepsilon \, \norm{h}_{L^1(\Omega)} \, \norm{q_{t,1} + v_t}_{L^1(\Omega)}
		+
		C \, \norm{q_{t,2}}_{L^2(\Omega)} \, \norm{\tau_t \, t \, (q_{t,1} + v_t)}_{L^2(\Omega)} \, \norm{h}_{L^1(\Omega)}
		.
	\end{equation*}
	Taking the supremum w.r.t.\ all $h \in L^1(\Omega)$
	with $\norm{h}_{L^1(\Omega)} \le 1$
	leads to the desired estimate
	$\norm{r(t)}_Y = \oo(t)$.
	Hence,
	\cref{assumption:sensitivity}
	holds.

	Next, we check that
	the requirement \ref{item:sufficient_abstract_1}
	of \cref{thm:sufficient_abstract} is satisfied.
	The \mbox{weak-$\star$} second-order epi-differentiability of $j = \delta_{\Uad}$
	was proved in \cref{thm:curvature}.
	The compactness assumption on $A_x = F''(\bar u)$
	was shown in \cref{lem:diff_bang_bang}.

	Finally, we study the linearized VI \eqref{eq:VI_derivative}.
	In our setting, it reads
	\begin{equation}
		\label{eq:linearized_VI_bang_bang}
		\dual{J_{up}(0, \bar u) p + J_{uu}(0,\bar u) \, \mu}{\zeta - \mu}
		+
		\frac12 Q\jxp(\zeta)
		-
		\frac12 Q\jxp(\mu)
		\ge
		0
		\qquad\forall \zeta \in \KK\jxp,
	\end{equation}
	where $J_{up}$, $J_{uu}$ denote the second partial derivatives of $J$.
	Thus, it remains to show that this linearized VI
	has at most one solution
	and that this solution is given by the solution of \eqref{eq:bang_bang_sensitivity_equation}.
	The condition \eqref{eq:explicitbangbang} implies that \eqref{eq:bang_bang_sensitivity_equation}
	possesses at most one solution.
	Hence, it is sufficient to prove that every solution $\mu \in \KK\jxp$
	of \eqref{eq:linearized_VI_bang_bang} also solves \eqref{eq:bang_bang_sensitivity_equation}.
	By using $\hat\zeta = \mu \pm t \, \zeta$
	and
	\begin{equation*}
		\frac12 \, Q\jxp\bigh(){\mu \pm t \, \zeta} - \frac12 \, Q\jxp(\mu)
		=
		\pm t \, \hat Q\jxp[\mu, \zeta] + \frac{t^2}{2} \, Q\jxp(\zeta),
	\end{equation*}
	this follows immediately.
	Applying \cref{thm:sufficient_abstract} finishes the proof.
\end{proof}

Note that a sequence $\bar u_t$ of stationary points with the properties in \cref{thm:sensitivity_bang_bang} can always be found by \cref{thm:lipschitz_solvability_bang_bang}.

We remark that differentiability results
for bang-bang optimal control problems governed by
ordinary differential equations can be found frequently in the literature.
See, e.g., \cite{Felgenhauer2010} and \cite{KimMaurer2003} for some examples. 
The result in \cref{thm:sensitivity_bang_bang}, however,
seems to be new and is, at least in the authors' opinion,
quite remarkable as it allows to precisely track how the 
sensitivity of a bang-bang solution $\bar u$ 
is related to the curvature properties of the set $\Uad  \subset \MM(\Omega)$
and the no-gap optimality condition in \cref{thm:ssc}.

\section{Concluding Remarks}
\label{sec:conclusion}

As we have demonstrated in \cref{sec:examples,sec:elliptic,sec:tangible_corollary},
our approach to the sensitivity analysis of variational 
inequalities 
allows not only to recover and extend known results
as those of \cite{Mignot1976,Haraux1977} and \cite{Do1992},
cf.\ \cref{cor:extendedpolyhedric} and \cref{th:elliptic},
but also to tackle problems that 
are beyond the scope of the classical theory, cf.\ the examples 
in \cref{subsec:bang_bang,subsec:prox_regular,subsec:plasti}. We hope that, because of the generality of 
our theorems and the self-containedness and elementary nature
of our proofs, our results will prove helpful 
to all those who are interested in the differentiability properties 
of solution operators to VIs of the first and the second kind
and the optimal control of variational inequalities. 

\ifbiber
	\renewcommand*{\bibfont}{\small}
	\printbibliography
\else
	\bibliographystyle{plainnat}
	\bibliography{references}
\fi
\end{document}